%% file: hoved.tex
\documentclass[12pt]{amsart}
\usepackage{amssymb,latexsym, mathtools,tikz}
\usetikzlibrary{arrows,decorations.markings}
\tikzstyle arrowstyle=[scale=1]
\tikzstyle directed=[postaction={decorate,
decoration={markings,mark=at position .65 with {\arrow[arrowstyle]{stealth}}}}]
\usepackage{blkarray}
\usepackage{xypic}
\usepackage[top=35mm, bottom=25mm, left=30mm, right=30mm]{geometry}

\usepackage[utf8]{inputenc}

\begin{document}

\title[Polarizations of powers of graded maximal ideals]{Polarizations of powers of
  graded maximal ideals}
\author{Ayah Almousa}
\address{Cornell University}
\email{aka66@cornell.edu}
\urladdr{http://math.cornell.edu/~aalmousa}
\thanks{AA was partially supported by the NSF GRFP under Grant No. DGE-1650441.}

\author{Gunnar Fl{\o}ystad}
\address{Universitetet i Bergen}
\email{gunnar@mi.uib.no}

\author{Henning Lohne}
\address{Universitetet i Bergen}
\email{lohne.henning@gmail.com}


\date{\today}
\begin{abstract}
We give a complete combinatorial characterization of all possible polarizations of powers of the graded maximal ideal $(x_1,x_2,\dotsc,x_m)^n$ of a polynomial ring in $m$ variables. We also give a combinatorial description of the Alexander duals of such polarizations. In the three variable case $m=3$ and also in the power two case $n=2$ the descriptions are easily visualized and 
we show that every polarization defines a (shellable) simplicial ball.
We give conjectures relating to topological properties and to
algebraic geometry, in particular that any polarization of an Artinian monomial ideal defines a simplicial ball.
\end{abstract}
\maketitle


\theoremstyle{plain}
\newtheorem{theorem}{Theorem}[section]
\newtheorem{corollary}[theorem]{Corollary}
\newtheorem*{main}{Main Theorem}
\newtheorem{lemma}[theorem]{Lemma}
\newtheorem{proposition}[theorem]{Proposition}
\newtheorem{conjecture}[theorem]{Conjecture}

\theoremstyle{definition}
\newtheorem{definition}[theorem]{Definition}
\newtheorem{fact}[theorem]{Fact}
\newtheorem{question}[theorem]{Question}
\newtheorem{problem}[theorem]{Problem}
\newtheorem{consequence}[theorem]{Consequence}

\theoremstyle{remark}
\newtheorem{notation}[theorem]{Notation}
\newtheorem{remark}[theorem]{Remark}
\newtheorem{example}[theorem]{Example}
\newtheorem{claim}{Claim}
\newtheorem{situation}[theorem]{Situation}


\newcommand{\psp}[1]{{{\bf P}^{#1}}}
\newcommand{\psr}[1]{{\bf P}(#1)}
\newcommand{\op}{{\mathcal O}}
\newcommand{\opw}{\op_{\psr{W}}}

\newcommand{\ini}[1]{\text{in}(#1)}
\newcommand{\gin}[1]{\text{gin}(#1)}
\newcommand{\kr}{{\Bbbk}}
\newcommand{\pd}{\partial}
\newcommand{\vardel}{\partial}
\renewcommand{\tt}{{\bf t}}


\newcommand{\coh}{{{\text{{\rm coh}}}}}


\newcommand{\modv}[1]{{#1}\text{-{mod}}}
\newcommand{\modstab}[1]{{#1}-\underline{\text{mod}}}

\newcommand{\sut}{{}^{\tau}}
\newcommand{\sumit}{{}^{-\tau}}
\newcommand{\til}{\thicksim}

\newcommand{\totp}{\text{Tot}^{\prod}}
\newcommand{\dsum}{\bigoplus}
\newcommand{\dprod}{\prod}
\newcommand{\lsum}{\oplus}
\newcommand{\lprod}{\Pi}

\newcommand{\La}{{\Lambda}}

\newcommand{\sirstj}{\circledast}

\newcommand{\she}{\EuScript{S}\text{h}}
\newcommand{\cm}{\EuScript{CM}}
\newcommand{\cmd}{\EuScript{CM}^\dagger}
\newcommand{\cmri}{\EuScript{CM}^\circ}
\newcommand{\cler}{\EuScript{CL}}
\newcommand{\clerd}{\EuScript{CL}^\dagger}
\newcommand{\clerri}{\EuScript{CL}^\circ}
\newcommand{\gor}{\EuScript{G}}
\newcommand{\cF}{\mathcal{F}}
\newcommand{\cG}{\mathcal{G}}
\newcommand{\cM}{\mathcal{M}}
\newcommand{\cE}{\mathcal{E}}
\newcommand{\cI}{\mathcal{I}}
\newcommand{\cP}{\mathcal{P}}
\newcommand{\cK}{\mathcal{K}}
\newcommand{\cS}{\mathcal{S}}
\newcommand{\cC}{\mathcal{C}}
\newcommand{\cO}{\mathcal{O}}
\newcommand{\cJ}{\mathcal{J}}
\newcommand{\cU}{\mathcal{U}}
\newcommand{\cQ}{\mathcal{Q}}
\newcommand{\cX}{\mathcal{X}}
\newcommand{\mm}{\mathfrak{m}}

\newcommand{\dlim} {\varinjlim}
\newcommand{\ilim} {\varprojlim}

\newcommand{\CM}{\text{CM}}
\newcommand{\Mon}{\text{Mon}}


\newcommand{\Kom}{\text{Kom}}


\newcommand{\EH}{{\mathbf H}}
\newcommand{\res}{\text{res}}
\newcommand{\Hom}{\text{Hom}}
\newcommand{\inhom}{{\underline{\text{Hom}}}}
\newcommand{\Ext}{\text{Ext}}
\newcommand{\Tor}{\text{Tor}}
\newcommand{\ghom}{\mathcal{H}om}
\newcommand{\gext}{\mathcal{E}xt}
\newcommand{\id}{\text{{id}}}
\newcommand{\im}{\text{im}\,}
\newcommand{\codim} {\text{codim}\,}
\newcommand{\resol}{\text{resol}\,}
\newcommand{\rank}{\text{rank}\,}
\newcommand{\lpd}{\text{lpd}\,}
\newcommand{\coker}{\text{coker}\,}
\newcommand{\supp}{\mathrm{supp}\,}
\newcommand{\Ad}{A_\cdot}
\newcommand{\Bd}{B_\cdot}
\newcommand{\Fd}{F_\cdot}
\newcommand{\Gd}{G_\cdot}


\newcommand{\sus}{\subseteq}
\newcommand{\sups}{\supseteq}
\newcommand{\pil}{\rightarrow}
\newcommand{\vpil}{\leftarrow}
\newcommand{\rpil}{\leftarrow}
\newcommand{\lpil}{\longrightarrow}
\newcommand{\inpil}{\hookrightarrow}
\newcommand{\pils}{\twoheadrightarrow}
\newcommand{\projpil}{\dashrightarrow}
\newcommand{\dotpil}{\dashrightarrow}
\newcommand{\adj}[2]{\overset{#1}{\underset{#2}{\rightleftarrows}}}
\newcommand{\mto}[1]{\stackrel{#1}\longrightarrow}
\newcommand{\vmto}[1]{\stackrel{#1}\longleftarrow}
\newcommand{\mtoelm}[1]{\stackrel{#1}\mapsto}
\newcommand{\bihom}[2]{\overset{#1}{\underset{#2}{\rightleftarrows}}}
\newcommand{\eqv}{\Leftrightarrow}
\newcommand{\impl}{\Rightarrow}

\newcommand{\iso}{\cong}
\newcommand{\te}{\otimes}
\newcommand{\into}[1]{\hookrightarrow{#1}}
\newcommand{\ekv}{\Leftrightarrow}
\newcommand{\equi}{\simeq}
\newcommand{\isopil}{\overset{\cong}{\lpil}}
\newcommand{\equipil}{\overset{\equi}{\lpil}}
\newcommand{\ispil}{\isopil}
\newcommand{\vvi}{\langle}
\newcommand{\hvi}{\rangle}
\newcommand{\susneq}{\subsetneq}
\newcommand{\sgn}{\text{sign}}


\newcommand{\xd}{\check{x}}
\newcommand{\ortog}{\bot}
\newcommand{\tL}{\tilde{L}}
\newcommand{\tM}{\tilde{M}}
\newcommand{\tH}{\tilde{H}}
\newcommand{\tvH}{\widetilde{H}}
\newcommand{\tvh}{\widetilde{h}}
\newcommand{\tV}{\tilde{V}}
\newcommand{\tS}{\tilde{S}}
\newcommand{\tT}{\tilde{T}}
\newcommand{\tR}{\tilde{R}}
\newcommand{\tf}{\tilde{f}}
\newcommand{\ts}{\tilde{s}}
\newcommand{\tp}{\tilde{p}}
\newcommand{\tr}{\tilde{r}}
\newcommand{\tfst}{\tilde{f}_*}
\newcommand{\empt}{\emptyset}
\newcommand{\bfa}{{\mathbf a}}
\newcommand{\bfb}{{\mathbf b}}
\newcommand{\bfd}{{\mathbf d}}
\newcommand{\bfl}{{\mathbf \ell}}
\newcommand{\bfx}{{\mathbf x}}
\newcommand{\bfm}{{\mathbf m}}
\newcommand{\bfv}{{\mathbf v}}
\newcommand{\bft}{{\mathbf t}}
\newcommand{\bbfa}{{\mathbf a}^\prime}
\newcommand{\la}{\lambda}
\newcommand{\bfen}{{\mathbf 1}}
\newcommand{\bfe}{{\mathbf 1}}
\newcommand{\ep}{\epsilon}
\newcommand{\en}{r}
\newcommand{\tu}{s}
\newcommand{\Sym}{\text{Sym}}
\newcommand{\uchi}{{\overline \chi}}
\newcommand{\fus}{{\mathbf 0}}

\newcommand{\ome}{\omega_E}

\newcommand{\bevis}{{\bf Proof. }}
\newcommand{\demofin}{\qed \vskip 3.5mm}
\newcommand{\nyp}[1]{\noindent {\bf (#1)}}
\newcommand{\demo}{{\it Proof. }}
\newcommand{\demodone}{\demofin}
\newcommand{\parg}{{\vskip 2mm \addtocounter{theorem}{1}  
                   \noindent {\bf \thetheorem .} \hskip 1.5mm }}

\newcommand{\lcm}{{\mathrm{lcm}}}


\newcommand{\dl}{\Delta}
\newcommand{\cdel}{{C\Delta}}
\newcommand{\cdelp}{{C\Delta^{\prime}}}
\newcommand{\dlst}{\Delta^*}
\newcommand{\Sdl}{{\mathcal S}_{\dl}}
\newcommand{\lk}{\text{lk}}
\newcommand{\lkd}{\lk_\Delta}
\newcommand{\lkp}[2]{\lk_{#1} {#2}}
\newcommand{\del}{\Delta}
\newcommand{\delr}{\Delta_{-R}}
\newcommand{\dd}{{\dim \del}}
\newcommand{\Del}{\Delta}

\renewcommand{\aa}{{\bf a}}
\newcommand{\bb}{{\bf b}}
\newcommand{\cc}{{\bf c}}
\newcommand{\xx}{{\bf x}}
\newcommand{\yy}{{\bf y}}
\newcommand{\zz}{{\bf z}}
\newcommand{\mv}{{\xx^{\aa_v}}}
\newcommand{\mF}{{\xx^{\aa_F}}}

\newcommand{\Symm}{\text{Sym}}
\newcommand{\pnm}{{\bf P}^{n-1}}
\newcommand{\opnm}{{\go_{\pnm}}}
\newcommand{\ompnm}{\omega_{\pnm}}

\newcommand{\pn}{{\bf P}^n}
\newcommand{\hele}{{\mathbb Z}}
\newcommand{\nat}{{\mathbb N}}
\newcommand{\rasj}{{\mathbb Q}}
\newcommand{\bfone}{{\mathbf 1}}

\newcommand{\dt}{\bullet}
\newcommand{\disk}{\scriptscriptstyle{\bullet}}

\newcommand{\cxF}{F_\dt}
\newcommand{\pol}{f}

\newcommand{\Rn}{{\mathbb R}^n}
\newcommand{\An}{{\mathbb A}^n}
\newcommand{\frg}{\mathfrak{g}}
\newcommand{\PW}{{\mathbb P}(W)}

\newcommand{\pos}{{\mathcal Pos}}
\newcommand{\g}{{\gamma}}

\newcommand{\Vaa}{V_0}
\newcommand{\Bp}{B^\prime}
\newcommand{\Bpp}{B^{\prime \prime}}
\newcommand{\bbp}{\mathbf{b}^\prime}
\newcommand{\bbpp}{\mathbf{b}^{\prime \prime}}
\newcommand{\bp}{{b}^\prime}
\newcommand{\bpp}{{b}^{\prime \prime}}

\newcommand{\oLa}{\overline{\Lambda}}
\newcommand{\ov}[1]{\overline{#1}}
\newcommand{\ovv}[1]{\overline{\overline{#1}}}
\newcommand{\tm}{\tilde{m}}
\newcommand{\po}{\bullet}

\newcommand{\surj}[1]{\overset{#1}{\twoheadrightarrow}}
\newcommand{\Supp}{\mathrm{supp}}
\newcommand{\bfw}{{\mathbf w}}
\newcommand{\bff}{\mathbf f}
\newcommand{\bfn}{\mathbf n}
\newcommand{\obfa}{\overline{\bfa}}
\newcommand{\nn}{{\mathbf n}}
\newcommand{\bfc}{{\mathbf c}}
\DeclarePairedDelimiter\abs{\lvert}{\rvert}%
\newcommand{\ayah}[1]{{\color{magenta} \sf AYAH: [#1]}}
\newcommand{\var}{\text{Var}}
\newcommand{\MM}{{\mathcal M}}
\newcommand{\ben}{{\mathbf 1}}

\newcommand{\bX}{{\mathbf X}}
\newcommand{\bY}{{\mathbf Y}}
\newcommand{\bZ}{{\mathbf Z}}

\newcommand{\Xv}{\check{X}}
\newcommand{\Yv}{\check{Y}}
\newcommand{\Zv}{\check{Z}}
\newcommand{\Xvp}{\check{X}^{\prime}}
\newcommand{\LS}{\mathrm{LS}}

\def\CC{{\mathbb C}}
\def\GG{{\mathbb G}}
\def\ZZ{{\mathbb Z}}
\def\NN{{\mathbb N}}
\def\RR{{\mathbb R}}
\def\OO{{\mathbb O}}
\def\QQ{{\mathbb Q}}
\def\VV{{\mathbb V}}
\def\PP{{\mathbb P}}
\def\EE{{\mathbb E}}
\def\FF{{\mathbb F}}
\def\AA{{\mathbb A}}

\newcommand{\oR}{\overline{R}}
\newcommand{\bfu}{{\mathbf u}}

\section{Introduction}

The notion of polarization of a monomial ideal has been a standard technique
in algebraic geometry and combinatorial commutative algebra for many years.
It gives a way to ``separate'' the powers of variables in a monomial ideal
while keeping
its homological properties. An early success was Hartshorne's proof the
connectedness of the Hilbert scheme \cite{Ha} using distractions, which are essentially
a form of polarization. Polarization gives a way to make combinatorial objects
from any monomial ideal: it gives a squarefree monomial ideal and thus a
simplicial complex, from which one may obtain algebraic invariants such as graded
Betti numbers in a combinatorial/topological way.
Although many authors have considered and used polarizations in various forms,
there has been no focused, systematic study of the variety of polarizations,
and of the characterizing properties of such ideals.
The purpose of this article is to undertake this task.
As we shall see, this raises both interesting questions in itself, as well
as connections to combinatorial topology and to algebraic geometry manifesting
in quite general conjectures.

\medskip
Before proceeding let us briefly give simple examples and explain the
notion of polarization.
Consider the ideal
 \begin{equation} \label{eq:intro-Ito}
 I = (x_1,x_2)^2 = (x_1^2, x_1x_2,x_2^2) \sus k[x_1, x_2].
 \end{equation}
 It polarizes to the ideal 
 \[ J = (x_{11}x_{12}, x_{11}x_{21}, x_{21}x_{22}) \sus
 k[x_{11},x_{12},x_{21},x_{22}]. \]
 The quotient ring $k[x_1,x_2]/I$ then comes from $k[x_{11},x_{12},x_{21},x_{22}]/J$
 by cutting down by a regular sequence of variable differences 
 \[ x_{11} - x_{12}, x_{21} - x_{22}. \]
 These two graded rings have the same homological properties, such as
 codimension and codepth, and the same graded Betti numbers.
 
 In general, for a monomial ideal $I$ one gets the polarization $J$ by 
 taking each minimal generator of $I$
 \begin{equation} \label{eq:intro-mon} x_1^{a_1}x_2^{a_2} \cdots x_m^{a_m} 
 \end{equation} and making a minimal 
 generator 
 \[ (x_{11}x_{12}\cdots x_{1a_1}) \cdot (x_{21}x_{22} \cdots x_{2a_2}) \cdot \cdots \cdot (x_{m1}\cdots x_{ma_m}) \]
 of $J$. We call this the {\it standard polarization}.
 
 However, for a monomial ideal $I$ in $k[x_1, \ldots x_m]$, there may be many other ways to get a squarefree monomial ideal $J$ in a larger polynomial ring $k[x_{ij}]$ such
 that $k[x_i]/I$ comes from $k[x_{ij}]/J$ by dividing out by a regular
 sequence of variable differences $x_{i,j} - x_{i,j^\prime}$.
 
 For instance, if $I$ is a strongly stable ideal, one has the ``b-polarization'', 
 \cite{NR09}, \cite{Yan12} which
 from the ideal \eqref{eq:intro-Ito} constructs the ideal
 \[ J^b = (x_{11}x_{12}, x_{11}x_{22}, x_{21}x_{22}) \sus
 k[x_{11},x_{12},x_{21},x_{22}]. \]
 (In this special case the ideals $J$ and $J^b$ are isomorphic, but this is not
 so in general.) The b-polarization takes the minimal generator \eqref{eq:intro-mon}
 and makes a minimal generator
 \[ (x_{1,1} \cdots x_{1,a_1}) \cdot (x_{2,a_1+1} \cdots x_{2,a_1+a_2}) \cdot \cdots 
 \cdot (x_{m,a_1 + \cdots + a_{m-1} +1} \cdots x_{m,a_1 + \cdots + a_m}) \]
 so the second index runs trough the integers from  $1$ to $a_1 + \cdots + a_m$.

 The notion of shifting operator is a form of polarization. It has been studied
 by several authors, \cite{Mur07} in the setting of squeezed spheres, and
 in \cite{Yan12}. 
 Recently letterplace ideals
 associated to poset ideals of $P$-partitions, \cite{Flo19} and \cite{FGH17}
 have provided large classes which are new polarizations of Artinian monomial
 ideals.

 \medskip
 \subsection{Polarizations of powers of the graded
   maximal ideal.}
 We combinatorially
 describe all possible polarizations of powers of the graded maximal ideal 
 \begin{equation} \label{eq:intro-I}
 I = (x_1, \ldots, x_m)^n \sus k[x_1, \ldots, x_m].
 \end{equation} 
 
 These maximal ideal powers are the ideals in
 $k[x_1, \ldots, x_m]$ with maximal symmetry. They are $GL(n)$-invariant. 
 A polarization is somehow a way of breaking this symmetry, but still keeping
 the homological properties. Thus we classify all symmetry breaks of this
 ideal. (Let us also mention that powers of graded maximal ideals have been studied
 from another combinatorial perspective in \cite{Wel02}, using discrete Morse theory to make
 cellular resolutions.)

   Let $\Delta_m(n)$ be the lattice simplex with vertices given by all $\bfa = 
   (a_1, \ldots, a_m) \in \NN_0^m$ with $\sum_{i=1}^m a_i = n$. Its vertices
   are in one-one correspondence with minimal generators of $(x_1, \ldots, x_m)^n$. See Figure \ref{fig:intro-D33} for $\Delta_3(3)$. It is divided into smaller triangles, and of these triangles, three
   are pointing down; we refer to these as {\it down-triangles}.
\begin{figure}
\begin{tikzpicture}
\draw (-3,0)--(3,0);
\draw (-2,1.6)--(2,1.6);
\draw (-1,3.2)--(1,3.2);
\draw (-3,0)--(0,4.8);
\draw (3,0)--(0,4.8) ;
\draw (-1,0)--(1,3.2);
\draw (1,0)--(-1,3.2);
\draw (-1,0)--(-2,1.6);
\draw (1,0)--(2,1.6);
\filldraw[black] (0,4.8) circle (2pt)  node[anchor=south] at (0,4.9){(3,0,0)};
\filldraw[black] (-3,0) circle (2pt)  node[anchor=east] at (-3.1,0){(0,3,0)};
\filldraw[black] (3,0) circle (2pt)  node[anchor=west] at (3.1,0){(0,0,3)};
\end{tikzpicture}
\caption{}
\label{fig:intro-D33}
\end{figure}
    By going upwards along the edges in this diagram, Figure \ref{fig:intro-D33}, we get a partial order $\geq_1$ on $\Delta_m(n)$. If the monomial \eqref{eq:intro-mon} polarizes to a monomial \[ m(\bfa) = m_1(\bfa) \cdots m_m(\bfa) \]
   where $m_i(\bfa)$ maps to $x_i^{a_i}$, let $X_1(\bfa)$ be the variables in 
   $m_1(\bfa)$. Then $X_1$ can be considered as a map from $\Delta_m(n)$ to the 
   boolean poset of subsets of the $x_{1j}$-variables. Similarly, we get maps $X_i$ for each $i = 1, \ldots, m$. For any polarization $J$ of $I$ in \eqref{eq:intro-I}, it turns out that
   each $X_i$ is an isotone map for a partial order $\geq_i$ on $\Delta_m(n)$.
   Conversely, given the maps $\{X_i\}$, we can construct monomial generators 
   $m(\bfa)$ of an ideal $J$. When is this a polarization of $I$?
   
   For a given edge in $\Delta_m(n)$ 
   between $\bfa$ and $\bfb$, we call the edge a {\it linear syzygy edge}
   if there is a linear syzygy between the monomials $m(\bfa)$ and $m(\bfb)$. 
   Our main result, Theorem \ref{thm:LS-XD}, says that the maps $\{X_i\}$ determine a polarization of $I$ if and only if the 
   {\it linear syzygy} edges for these maps contain a spanning tree of the
   edge graph (this is a complete graph) of each
  higher-dimensional ``down-triangle" of $\Delta_m(n)$.
   
    Two special cases are worth attention due to the easy visualization of the various polarizations. First, the polarizations of $(x_1,x_2,x_3)^n$, which are
    easily visualized in terms of Figure \ref{fig:intro-D33}, see Corollary
    \ref{cor:LS-m3} and Example \ref{eks:LS-m3}. Secondly, the polarizations of
    $(x_1,x_2, \ldots, x_m)^2$, which are in one-to-one correspondence with trees on $(m+1)$ vertices, Theorem \ref{thm:degto-T} and Examples
    \ref{ex:degto-fire} and \ref{ex:degto-five}.

\medskip
\subsection{Polarizations of Artinian monomial ideals.}
A nice result on the topological properties of polarizations by
S.Murai \cite{Mur11} is that the standard polarization of any
Artinian monomial ideal gives a simplicial ball. The boundary is a simplicial
sphere with a simply described Stanley-Reisner ideal,
which generalizes Bier spheres, \cite{BZ}. This
construction has again been comprehensively generalized by A. D'Ali,
A.Nematbakhsh and the second author in \cite{AFN19}, by considering letterplace
ideals. 

We here give the full simple conjecture, Conjecture \ref{con:conj-ball},
that via the Stanley-Reisner correspondence, any polarization
of an Artinian monomial ideal gives a simplicial ball
with a simple and naturally
described Stanley-Reisner ideal of its boundary, a simplicial sphere.
We show that all polarizations of $(x_1, x_2, x_3)^n$ 
and of $(x_1, \ldots, x_m)^2$
give simplicial balls, by showing that the Alexander duals of all such polarizations have linear quotients.
In a forthcoming paper we show the
squeezed spheres of \cite{Ka} and the generalized Bier spheres both fit into one
and the same framework of polarization and deseparation which we give here, see Section \ref{sec:sep}.


Relating to algebraic geometry
we conjecture, Conjecture \ref{conj:conj-defo}, that every first order deformation of any polarization of any Artinian monomial ideal lifts to a global deformation. In particular, such a polarization would be a smooth point on
any Hilbert scheme. If true,
this would give a plethora of components of the Hilbert scheme even for the
same Hilbert polynomial.

\medskip
\subsection{Alexander duals of polarizations.} If $J$ is any polarization
of an Artinian monomial ideal $I$ of $k[x_1, \ldots, x_m]$, then its Alexander
dual $J^\vee$ will be generated by ``colored'' monomials of degree $m$, monomials
of the form 
\[ x_{1j_1}x_{2j_2} \cdots x_{mj_m} \]
where for each $i = 1, \ldots, m$, we have one variable $x_{i,j_i}$ from the class of $i$-variables (color $i$). We call these {\it rainbow monomials}. 
The class of ideals generated by rainbow monomials and with 
$m$-linear resolution is precisely the class which is Alexander dual to the
class of polarizations of Artinian monomial ideals in $m$ variables, Proposition
\ref{pro:conjRainbow}. A concise criterion for ideals generated by rainbow
monomials to have linear resolution is given by A.Nematbakhsh in \cite{Nem}, and
we recall it in Theorem \ref{thm:conjRainbow}.

We describe the Alexander dual of any polarization of
the maximal ideal power $(x_1, \ldots, x_m)^n$, in terms of the isotone maps
$X_i$. As it turns out, the argument involves only a baby-form $\chi_i$ of the isotone maps $X_i$, where $\chi_i$ is an isotone map from $\Delta_m(n)$ to the poset $\{0 < 1\}$.
This description of the Alexander dual, however, leaves something to be desired concerning transparency; for instance, it is not obvious that the number of generators is actually always
$\binom{n+m-1}{m}$. We discuss problems concerning ideals generated by rainbow monomials in Section \ref{sec:conj}.

\medskip
\subsection{Organization of the paper.}
In Section \ref{sec:sep} we recall the notions of separations, separated models and polarizations. We show some basic results on polarizations of Artinian monomial ideals and introduce the isotone maps $X_i$.
In Section \ref{sec:conj} we discuss various conjectures and problems that have come up
during our investigations.
Section \ref{sec:lin} contains our main result concerning the complete
combinatorial classification of all polarizations of $(x_1, \ldots, x_m)^n$.
Examples are given for the three-variable case.
A short Section \ref{sec:trevar} classifies for the three-variable the isotone maps $\Delta_3(n) \pil B(n)$.
In Section \ref{sec:degto} we consider the degree two case and show that polarizations of 
$(x_1, \ldots, x_m)^2$ are in one-to-one correspondence with trees on $(m+1)$ vertices.
Section \ref{sec:AD} gives the Alexander duals of polarizations of maximal ideal powers, and in Section 
\ref{sec:LQ} we show that polarizations of $(x_1,x_2,x_3)^m$ are shellable, which implies that they define simplicial balls by the Stanley-Reisner correspondence.

\medskip
\noindent {\it Note.} The results in Section
\ref{sec:degto}, Section \ref{sec:trevar}, and Theorem \ref{thm:LS-XD} in the special case of three variables are essentially found in the unpublished preprint \cite{Loh13} by the 
third author. 

\section{Separations of monomial ideals}
\label{sec:sep}
We recall the basic notions of separation of a monomial ideals and  separated models, as introduced in \cite{FGH17}. We also define a polarization of 
a monomial ideal as a separation which is a squarefree monomial ideal. 
   We consider Artinian monomial ideals and show that for these the notion of polarization and separated models are the same. We introduce the isotone maps
   $X_i$ from the lattice simplex $\Delta_m(n)$ to the Boolean poset $B(n)$, which are our
   main gadgets to classify all the polarizations of maximal ideal powers.
   
\medskip
If $R$ is a set, let $k[x_R]$ be the polynomial ring in the variables 
$x_r$ where $r \in R$. If $S \pil R$ is a map of sets, it induces a $k$-algebra
homomorphism $k[x_S] \pil k[x_R]$ by mapping $x_s$ to $x_r$ if $s \mapsto r$.

\subsection{Separations and polarizations}
\begin{definition}
Let $R^\prime \mto{p} R$ be a surjection of finite sets with the
cardinality of $R^\prime$ one more than that of $R$. Let $r_1$ and $r_2$
be the two distinct elements of $R^\prime$ which map to a single
element $r$ in $R$. Let $I$ be a monomial ideal in the polynomial ring
$k[x_R]$ and $J$ a monomial ideal in $k[x_{R^\prime}]$. We say
$J$ is a {\it simple separation} of $I$ if the following holds:
\begin{itemize}
\item[i.] The monomial ideal $I$ is the image of $J$ by the map
$k[x_{R^\prime}] \pil k[x_R]$.
\item[ii.] Both the variables $x_{r_1}$ and $x_{r_2}$ occur in some minimal
generators of $J$ (usually in distinct generators).
\item[iii.] The variable difference $x_{r_1} - x_{r_2}$ is a non-zero divisor
in the quotient ring $k[x_{R^\prime}]/J$.
\end{itemize}

More generally, if $R^\prime \mto{p} R$ is a surjection of finite sets
and $I \sus k[x_R]$ and $J \sus k[x_{R^\prime}]$ are monomial ideals such
that $J$ is obtained by a succession of simple separations of $I$,
$J$ is a {\it separation} of $I$.
 If $J$ has no further separation, we call $J$ a {\it separated model}
(of $I$).
\end{definition}

A simple, but for us significant observation is that the minimal
generators of the separation $J$ and the ideal $I$
are in one-one correspondence. More generally
the graded Betti numbers of $J$ and $I$ are the same, since we get
from $k[x_{R^\prime}]/J$ to $k[x_R]/I$ by dividing out by a regular sequence of linear forms.

In \cite{AHB16} it is shown that simple separations may be considered as 
deformations of the ideal $I$.

Any monomial ideal may be separated to its standard polarization. So clearly
any separated model is a squarefree monomial ideal. The standard polarization
may, however, be further separable, so it may not be a separated model.

\begin{example} \label{ex:seppol} Consider $I = (x^2y^2, x^2z^2, y^2z^2)$ in $k[x,y,z]$.
The standard polarization is
\[ \tilde{I} = (x_1x_2y_1y_2, x_1x_2z_1z_2, y_1y_2z_1z_2). \]
This may be further separated to
\[ J = (x_1x_2y_1y_2, x^\prime_1x^\prime_2z_1z_2, y_1y_2z_1z_2). \]
\end{example}

\begin{definition}
  Let $I \sus k[x_R]$ be a monomial ideal and $R^\prime \pil R$ be a surjection
  of finite sets. An ideal $J \sus k[x_{R^\prime}]$ is a {\it polarization}
  of $I$ if $J$ is squarefree and a separation of $I$.
\end{definition}

This general notion of polarization is likely first defined in \cite{Yan12}.
By the example above it is not true that any polarization is a
separated model. However, we shall see in Corollary \ref{cor:sep-artin}
that for Artinian monomial ideals, these notions are equivalent.

\medskip
We state a general lemma which will be useful later.

\begin{lemma} \label{lem:sepx01}
  Let $I$ be a monomial ideal in $k[x_0,x_1, \ldots, x_m]$
  such that each generator of $I$ is squarefree in the $x_0$-variable.
  Then if $(x_0 - x_1) \cdot f$ is in $I$, then for every monomial
  $m$ in $f$ we have that $x_0m$ and $x_1m$ are in $I$.
\end{lemma}

\begin{proof}
  Let $f= x_0^a f_a + x_0^{a-1}f_{a-1} + \cdots + f_0$ where each $f_p$
  have no $x_0$-terms. Then if
  $(x_0 - x_1)f$ is in $I$, the only terms with $x_0^{a+1}$ are the
  terms in $x_0^{a+1} f_a$, and so these are in $I$ since we are in a
  $\ZZ^m$-graded setting. But since $I$ is squarefree in $x_0$,
  we have $x_0f_a$ in $I$ and so $x_0^a f_a$ in $I$. In this way we may
  peel off and get that all terms $x_0^pf_p$ are in $I$ for $p \geq 1$.
  
  Then in $(x_0 - x_1) f_0$, the terms with $x_0$ are those in $x_0f_0$.
  Hence $x_0f_0$ is in $I$ and so $x_0 f$ is in $I$. Again since $I$ is multigraded, each 
  monomial term $x_0m$ is in $I$.
  We also get $x_1 f \in I$ and then each $x_1m \in I$.
  \end{proof}

\subsection{Polarizations of Artinian monomial ideals}
We consider an Artinian monomial ideal $I \sus k[x_1, \ldots, x_m]$.
This is simply a monomial ideal such that for every index $i$, some
power $x_i^{n_i}$ is a minimal generator of $I$.
Let $\Xv_i = \{ x_{i1}, x_{i2}, \ldots, x_{in_i^\prime} \}$ be a set of variables.
We get a polynomial ring whose variables
are those in the union of all these variables, and a homomorphism
\[ \pi : k[\Xv_1, \Xv_2, \ldots, \Xv_m] \pil k[x_1, \ldots, x_m], \]
by mapping every variable in $\Xv_i$ to $x_i$. 

In a polarization $J \sus k[\Xv_1, \ldots, \Xv_m]$ of $I$,
we separate each monomial generator
$x^{\bfa} = x_1^{a_1}x_2^{a_2} \cdots x_m^{a_m}$ to squarefree monomials
\begin{equation} \label{eq:sepma} m(\bfa) = m_1(\bfa) \cdot m_2(\bfa) \cdots m_m(\bfa),
\end{equation}
where $m_i(\bfa)$ is a squarefree monomial of degree $a_i$ in variables from $\Xv_i$. 
Considering $x_i^{n_i}$ we see that we must have $n_i^\prime \geq n_i$.
We shall shortly show, Remark \ref{rem:sep-ni}, that we may take $n_i^\prime = n_i$.

Starting from $K[\Xv_1, \ldots, \Xv_m]/J$ we get the quotient ring
$k[x_1, \ldots, x_m]/I$ be dividing out by a regular sequence consisting
of variable differences $x_{ip} - x_{iq}$. For each $i$ we choose $(n_i-1)$
linearly independent such variable differences.  Any such sequence of
variable differences in any order will do.

We may get intermediate separations of $I$ as follows. Choose
surjections $p_i : \Xv_i \pil \Xv_i^\prime$. We get a map of polynomial rings
\[  k[\Xv_1, \ldots, \Xv_m] \pil k[\Xv_1^\prime, \ldots, \Xv_m^\prime]. \]
The image of the polarization $J$ is an ideal $I^\prime$ in
$k[\Xv_1^\prime, \ldots, \Xv_m^\prime]$ and $I^\prime$ is a separation of $I$.
We then get from
$k[\Xv_1, \ldots, \Xv_m]/J$ to $k[\Xv^\prime_1, \ldots, \Xv^\prime_m]/I$ by dividing
out by a regular sequence of variable differences $x_{ia} - x_{ib}$ where
for each $i$, $x_{ia}$ and $x_{ib}$ are in the same fiber $p_i^{-1}(x^\prime)$
of $p_i$, and we have $|p_i^{-1}(x^\prime)|-1$ linearly independent such
variable differences for each fiber.

\medskip
The following lemma has some key consequences.

\begin{lemma} \label{lem:sep-artin}
  Let $x^\bfa$ and $x^\bfb$ be minimal generators of a monomial
  ideal $I$ and $m(\bfa)$ and $m(\bfb)$ the corresponding generators
  in a polarization of $I$.

  Fix an index $i$. If $a_i \leq b_i$ and $a_j \geq b_j$ for every $j \neq i$,
  then the $i$'th part $m_i(\bfa)$ divides $m_i(\bfb)$.
\end{lemma}

\begin{proof} We shall use induction on $d = b_i - a_i$. If $d = 0$ then
  clearly $\bfb = \bfa$ and there is nothing to prove. We may also assume that
  $a_i \geq 1$, since otherwise there is nothing to prove. We suppose
  $m_i(\bfa)$ does not divide $m_i(\bfb)$ and we shall derive a contradiction.
  If it does not divide we may factor $m_i(\bfb)$ as $t_i(\bfb) \cdot n_i(\bfb)$ where 
  $t_i(\bfb)$ has degree $d + 1$ and has no common variable with
  $m_i(\bfa)$. (We are of course using here that $m_i(\bfb)$ is squarefree.)
  For simplicity we may re-index variables so that $t_i(\bfb) =
  x_{i1}x_{i2} \cdots x_{i,d+1}$. We now in $k[\Xv_1, \ldots, \Xv_m]/J$ divide out by
  all the variable differences involving $\Xv_j$-variables where $j \neq i$, and
  by all variable differences $x_{ir}- x_{i,r+1}$ for $r = d+2, \ldots, n_i-1$.
  Thus we are collapsing all the $\Xv_j$-variables into the single variable $x_j$
  and the variables $x_{i,d+2}, \ldots, x_{i,n_i}$ into a single variable $x_i$.
  We get a quotient ring
  \begin{equation} \label{eq:sepQuotring}
    k[x_{i1}, \ldots, x_{i,d+1}, x_1, \ldots, x_m]/I^\prime,
    \end{equation}
    where $I^\prime$ is a separation of $I$. Note that $m(\bfa)$ collapses
    to $x^\bfa$ in $I^\prime$. 

Consider now the variable difference $x_{i,d+1} - x_i$ in the polynomial ring
above. We see that 
\begin{align} \notag  & (x_{i,d+1} - x_i) x_{i1} \cdots x_{id} \cdot x_i^{a_i-1}  \prod_{j \neq i}
  x_j^{a_j} \\ \label{eq:sep-xx}
=\, & x_{i1} \cdots x_{i,d+1} \cdot  x_i^{a_i-1}  \prod_{j \neq i} x_j^{a_j}- 
x_{i1} \cdots x_{i,d} \cdot  x_i^{a_i}  \prod_{j \neq i} x_j^{a_j}
\end{align}
vanishes in the quotient ring \eqref{eq:sepQuotring}: Tthe first term
is divisible by the image of $m(\bfb)$ in $I^\prime$ (note that $b_i = (d+1) + (a_i - 1)$), and the second term is
divisible by the image of $m(\bfa)$. Since $x_{i,d+1} - x_i$ is
not a zero divisor (it belongs to a regular sequence), we get from \eqref{eq:sep-xx} that 
\begin{equation} \label{eq:sepn} \bfn = x_{i1} \cdots x_{id}  \cdot x_i^{a_i-1}  \prod_{j \neq i}x_j^{a_j}
\end{equation} 
is in $I^\prime$.
Now if $d = 1$, this monomial has $\ZZ^m$-degree $\bfa$. But
the monomial $x^\bfa$ is in $I^\prime$, with the same degree. Since these
are the $\ZZ^m$-degree of a generator of $I$, there can only be a single monomial in $I^\prime$ with
this $\ZZ^m$-degree. We get a contradiction. Now suppose $d \geq 2$.
Then $\bfn$ is divisible by a generator $m^\prime(\bfc)$ in $I^\prime$ which
can {\it not} be $x^\bfa$. We will have each $c_j \leq a_j$ for $j \neq i$,
and so $c_i > a_i$. Furthermore, we have $b_i > c_i$ since $\bfn$ in \eqref{eq:sepn} has
$i$-degree $d + a_i -1 = b_i - 1$. By induction on $d$, considering the polarized ideal 
$J$, the $i$'th part $m_i(\bfa)$ here divides
the $i$'th part $m_i(\bfc)$. But then going to $I^\prime$ then $x_i^{a_i}$ divides
the image of $m^\prime_i(\bfc)$, and so  $x_i^{a_i}$ would divide $\bfn$ of \eqref{eq:sepn}, a contradiction.
\end{proof}

\begin{remark} \label{rem:sep-ni}
  If $m(\bfa)$ is a minimal generator of $J$, by the lemma
  $m_i(\bfa)$ will divide $m_i(0, \ldots, n_i, \ldots, 0)$ which of course
  is just $m(0, \ldots, n_i, \ldots, 0)$. Thus, if the polarization of
  $x_i^{n_i}$ is $x_{i1}x_{i2} \cdots x_{in_i}$, then every $x_i$-variable occurring
  in the minimal generators of $J$ are among these variables, and so
  we may take $\Xv_i = \{ x_{i1}, \ldots, x_{in_i} \}$.
\end{remark}

The following is quite particular for Artinian monomial ideals, note Example \ref{ex:seppol}.
\begin{corollary} \label{cor:sep-artin}
  Every polarization of an Artinian monomial ideal $I$ is a separated model
  for $I$.
  \end{corollary}

  \begin{proof}
    If the polarization $J$ was not a separated model, then let $J^\prime$ be
    a further simple separation. Since $I$ in $k[x_1, \ldots, x_m]$
    is an Artinian monomial ideal, every variable $x_i$ of course occurs
    in a minimal generator of $I$, in fact $x_i^{n_i}$ is a minimal generator.
    Then if $J^\prime$ is in $k[\Xv_1^\prime, \ldots, \Xv_m^\prime]$ then every
    variable in this polynomial ring must also occur in a generator of
    $J^\prime$, by the definition of a separation.
    By the above Lemma \ref{lem:sep-artin} and Remark \ref{rem:sep-ni}, if $x_i^{n_i}$ polarizes to
    $x_{i1} \cdots  x_{in_i}$ then $\Xv_i^\prime = \{x_{i1}, \ldots, x_{in_i} \}$.
    But $J$ is obtained from $J^\prime$ by dividing out by a variable
    difference $x_{ia} - x_{ib}$. Then the image of $x_{i1} \cdots x_{in_i}$
    in $J$ would not be squarefree, a contradiction.
    \end{proof}

\subsection{Polarizations of powers of the graded maximal ideal}
We now consider powers of the maximal ideals
\[ M= (x_1,x_2, \ldots, x_m)^{n} \sus k[x_1, \ldots, x_m].\]
A polarization of this ideal may by Remark \ref{rem:sep-ni} be taken
to live in polynomial ring
$k[\Xv_1, \ldots, \Xv_m]$ where $\Xv_i = \{ x_{i1}, \ldots, x_{in} \}$. 
Our goal is to combinatorially classify all possible polarizations
of $M$ in this polynomial ring.

The generators of the monomial ideal $M$ are  all monomials $x_1^{b_1} \ldots x_m^{b_m}$ with $b_1 + \cdots + b_m = n$.

\begin{definition} $\Delta_m(n)$ is the subset of $\NN_0^m$ of all 
  tuples $\bfb = (b_1, \ldots, b_m)$ of non-negative integers with
  $b_1 + \cdots + b_m = n$. For a given $\bfb$, its {\it support}
  $\supp \, (\bfb)$ is the set of all $i$ such that $b_i \geq 1$. 
\end{definition}

In a polarization $J$ of $M$ we have one minimal generator of $J$, $m(\bfb)$
for every $\bfb$ in $\Delta_m(n)$.
Now fix an index $1 \leq i \leq m$. Then $\Delta_m(n)$ may be given a partial order $\geq_i$ 
by letting $\bfb \geq_i \bfa $ if $b_i \geq a_i$ and $b_j
\leq a_j$ for $j \neq i$. Thus there is one maximal element
$(0,\ldots, 0, n, 0, \ldots, 0)$ where $n$ is in position $i$, and
it has minimal elements all $\bfb $ with $b_i = 0$.
This is a graded partial order with $\bfb$ of rank $b_i$. 

Now, given any $\bfb \in \Delta_m(n)$, we get from the polarization $J$ a squarefree monomial 
$m_i(\bfb)$, see \eqref{eq:sepma}.
The set of variables dividing this monomial is a subset of $\Xv_i$ which we denote as $X_i(\bfb)$. Let $B(\Xv_i)$ be the Boolean
poset on $\Xv_i$, a Boolean poset on a set of $n$ elements.
We get a function
\begin{align}  \label{eq:sep-iso} X_i : \Delta_m(n) & \pil B(\Xv_i) \\
 \notag \bfb & \mapsto X_i(\bfb).
\end{align}

The following is an immediate consequence of Lemma \ref{lem:sep-artin}.

\begin{corollary} \label{pro:sepIsotone}
  Let $J$ be a polarization of $M$. Then $X_i$ is an isotone
  rank-preserving map when $\Delta_m(n)$ has the ordering $\geq_i$.
\end{corollary}

\begin{remark} \label{rem:sep-group}
Since $\Xv_i = \{ x_{i1}, \ldots, x_{in}\}$, the group $S_n$ acts on $\Xv_i$.
Also the group $S_m$ acts on $k[x_1, \ldots, x_m]$ by permutation of variables and
hence on the set of maps $\{X_i\}$. In all there is an action of a semi-direct 
product $S_m \ltimes (S_n)^m$ on $k[\Xv_1, \ldots, \Xv_m]$ compatible with the
action of $S_m$ on $k[x_1, \ldots, x_m]$. Since $(x_1, \ldots, x_m)^n$ is equivariant
for the group action, the isomorphism classes of polarizations of this maximal ideal power are precisely the
orbits of $S_m \ltimes (S_n)^m$ on the set of polarizations. 
\end{remark}

The convex hull of  $\Delta_m(n)$ in $\RR^{m}$ is a simplex of dimension
$(m-1)$. We shall however only need the graph structure it induces.
Given a point $\bfc$ in $\Delta_m(n+1)$ and $i,j$ in the support of $\bfc$,
let $e_i$ and $e_j$ be the unit coordinate vectors. 
Then we get an edge between the points $\bfc - e_i$ and $\bfc - e_j$
in $\Delta_m(n)$, denoted $(\bfc;i,j)$.
Every edge in $\Delta_m(n)$ is of this form for
unique $\bfc, i$ and $j$. A point $\bfc$ of $\Delta_m(n+1)$  induces
a subgraph of $\Delta_m(n)$, the complete {\it down-graph} $D(\bfc)$
on the points $\bfc - e_i$ for $i \in \Supp \, (\bfc)$. 
The {\it dimension} of $D(\bfc)$ is the dimension of its convex hull,
which is one less than the cardinality of $\Supp \, (\bfc)$.
When
$\Supp \,(\bfc)$ has dimension two we call this a {\it down-triangle}. 
In Figure \ref{fig:sep-D33} we have three down-triangles.

\medskip
Let $\Delta_m^+(n+1)$ be the subset of $\Delta_m(n+1)$ consisting of
$\bfc$ with $c_i \geq 1$ for every $i$. The complete down-graph $D(\bfc)$
has  full (maximal) dimension $(m-1)$ iff $\bfc \in \Delta_m^+(n+1)$. Note
that $\Delta_m^+(n+1)$ is in one-one correspondence with $\Delta_m(n+1-m)$
by sending $\bfc$ to $\bfc - \ben$ where $\ben = (1,1, \ldots, 1)$. 
In Figure \ref{fig:sep-D33} there are three full-dimensional complete down-graphs,
or in this case down-triangles,
corresponding to the three elements of $\Delta_3(1)$, or equivalently 
of $\Delta_3^+(4)$. 

\begin{figure}
\begin{tikzpicture}
\draw (-3,0)--(3,0);
\draw (-2,1.6)--(2,1.6);
\draw (-1,3.2)--(1,3.2);
\draw (-3,0)--(0,4.8);
\draw (3,0)--(0,4.8) ;
\draw (-1,0)--(1,3.2);
\draw (1,0)--(-1,3.2);
\draw (-1,0)--(-2,1.6);
\draw (1,0)--(2,1.6);
\filldraw[black] (0,4.8) circle (2pt)  node[anchor=south] at (0,4.9){(3,0,0)};
\filldraw[black] (-3,0) circle (2pt)  node[anchor=east] at (-3.1,0){(0,3,0)};
\filldraw[black] (3,0) circle (2pt)  node[anchor=west] at (3.1,0){(0,0,3)};
\end{tikzpicture}
\caption{}
\label{fig:sep-D33}
\end{figure}


Each $\bfa$ in $\Delta_m(n-1)$ also determines a subgraph of $\Delta_m(n)$, 
the complete {\it up-graph} $U(\bfa)$ consisting of the points
$\bfa + e_i$ for $i = 1, \ldots, m$ and with edges
$(\bfa + e_i + e_j;i,j)$ for $i \neq j$. For each $\bfa$ in $\Delta_m(n-1)$
the convex hull of the up-graph $U(\bfa)$ has full dimension $(m-1)$.
When $m = 3$ we call this an {\it up-triangle}. 
In Figure \ref{fig:sep-D33} there are six up-triangles.

\section{Conjectures and problems} \label{sec:conj}

Before embarking on the main results of the paper we here
discuss conjectures and problems on polarizations of Artinian monomial ideals in general that have come up during our investigations. 
They concern i) the topology of their associated simplicial complexes, ii) their Alexander duals, a class of rainbow
monomial ideals (terminology introduced here), iii) their deformations.

\subsection{Balls and spheres} \label{subsec:conj-ball}
\begin{lemma} \label{lem:conj-codimone}
Let $\Delta(J)$ be the simplicial complex associated to the polarization $J$
of an Artinian monomial ideal $I$. Then every codimension one face of $\Delta(J)$
is contained in one or two facets. If $I$ is not a complete intersection, then at
least once there is a codimension one face contained in exactly one facet.
\end{lemma}

\begin{proof} Let $\Delta_i$ be the simplex on $\{(i,j) \, | \, j = 1, \ldots, n_i\}$.
The squarefree monomial $x_{i1}x_{i2}\cdots x_{in_i}$ in $k[\Xv_i]$ defines
the sphere which is the boundary $\vardel \Delta_i$ of this simplex. 
The natural polarization in $k[\Xv_1, \ldots, \Xv_m]$ of the complete intersection
$(x_1^{n_1}, x_2^{n_2}, \ldots, x_m^{n_m})$ then defines the sphere of codimension $m$
which is the join $S = \underset{i=1}{\overset{m}*} \vardel \Delta_i$. 
Every codimension one face is here on precisely two facets.

The simplicial complex $\Delta(J)$ is a Cohen-Macaulay subcomplex of $S$ with the 
same dimension as $S$. If $\Delta(J)$ is not all of $S$, let $F$ be a facet of $\Delta(J)$
and $G$ a facet of $S$ not in $\Delta(J)$. Since $S$ is strongly connected,
 there is a path of facets
\[ F = F_0, F_1, \ldots, F_r = G \]
such that $F_{i} \cap F_{i+1}$ has codimension one for each $i$, \cite[Prop.9.1.12]{HH11}.
Let $p$ be maximal such that $F_p$ is in $\Delta(J)$. Then
$F_p \cap F_{p+1}$ is only on the facet $F_p$ in $\Delta(J)$. 
\end{proof}

By a result of Bj\"orner \cite[Thm.11.4]{Bj95} a constructible simplicial complex with the property
of Lemma \ref{lem:conj-codimone} above
is a simplicial ball or a simplicial sphere (the latter when every codimension
one face is on exactly two facets). 

\begin{conjecture} \label{con:conj-ball} The simplicial complex $\Delta(J)$ associated to a polarization
$J$ of an Artinian monomial ideal $I$, is a simplicial ball, save for the case
when $I$ is a complete intersection, when it is a simplicial sphere.
\end{conjecture}

By the result of Bj\"orner loc.cit. a positive answer to the following question would settle the above conjecture.
\begin{question}
Do polarizations of Artinian monomial ideals have constructible (for instance shellable) simplicial complexes?
\end{question}

That the standard polarization of an Artinian monomial ideal is shellable
seems first to have been shown by A.Soleyman Jahan in \cite{Sol07}. In \cite{Mur11} S.Murai uses
this to conclude that the standard polarizations give simplicial balls. 
More generally it is shown that letterplace ideals 
define simplicial balls, \cite{AFN19}, by showing that these simplicial
complexes are shellable. 
Letterplace ideals are introduced in \cite{FGH17} and are
polarizations of Artinian monomial ideals. The article \cite{Flo19} discusses such Artinian monomial ideals more in depth.

In our last Section \ref{sec:LQ} we show in the case of {\it three} variables that the Alexander dual
of any polarization $J$ has linear quotients, see \cite[Sec.8.2, Cor.8.2.4]{HH11} for this notion. In Section \ref{sec:degto} we show when the power of the maximal
ideal (in any number of variables) is two, then the Alexander dual has linear quotients.
Thus in these cases the simplicial complex $\Delta(J)$ is shellable and hence a simplicial ball. 

\medskip
For a letterplace ideal the second author et.al. in \cite{AFN19} get an 
explicit simple description of the Stanley-Reisner ideal of the boundary
of the simplicial ball defined by the letterplace ideal. 
In fact a general result, see \cite[Section 5]{BH93}, says that the 
canonical module of a Stanley-Reisner ring $k[\Delta]$ identifies 
as a multigraded proper ideal of this ring $k[\Delta]$ if and only if $\Delta$ is a homology ball.
Then the ideal defines the boundary of this homology ball, which is a homology sphere.
In \cite{AFN19} an explicit description of this canonical module is given. 

\begin{conjecture}
For polarizations of Artinian monomial ideals, the  canonical module identifies
(in a simply described and natural way) as  
a multigraded ideal of the Stanley-Reisner ring of the polarization.
\end{conjecture} 

\begin{consequence} With Conjecture \ref{con:conj-ball}, this would give an explicit description of the Stanley-Reisner ring of the boundary of the simplicial ball, defining a simplicial sphere.
\end{consequence} 

\subsection{Rainbow monomial ideals with linear resolution}
Recall that for a squarefree monomial ideal $J$ in a polynomial ring $S$, the squarefree monomial ideal $I$ is the \textit{Alexander dual} of $J$ if the monomials in $I$ are precisely the monomials in $S$ which have nontrivial common divisor with every monomial in $J$, or equivalently, every generator of $J$.

We may consider each set of variables $\Xv_i$, $i = 1, \ldots, m$ as a 
{\it color}
class of monomials. A monomial $x_{1i_1}x_{2i_2} \cdots x_{mi_m}$
with one variable of each color is a {\it rainbow} monomial.

\begin{proposition} \label{pro:conjRainbow} The class of ideals generated by rainbow monomials and with 
$m$-linear resolution is precisely the class which is Alexander dual to the
class of polarizations of Artinian monomial ideals in $m$ variables:
\begin{enumerate}
\item[a.] Let $J$ be a polarization of an Artinian
monomial ideal $I$ in $k[x_1, \ldots, x_m]$.
The Alexander dual ideal of $J$ is generated by rainbow monomials
and has $m$-linear resolution.
\item[b.] If an ideal $J^\prime$ is generated by rainbow monomials and has $m$-linear resolution
(and every variable in the ambient ring occurs in some generator of the ideal),
then its Alexander dual $J$ is a polarization of an Artinian monomial ideal
in $m$ variables.
\end{enumerate}
\end{proposition}

\begin{proof}
a. Since  $I$ is Cohen-Macaulay of codimension $m$, the same is true for $J$.
Then the Alexander dual of $J$ is generated in degree $m$ and has
$m$-linear resolution \cite{ER98}. But if $\bfm$ is a generator for this
Alexander dual, it has a common variable with $x_{i1}x_{i2} \cdots x_{in_i}$
(the polarization of $x_i^{n_i}$) for every $i = 1, \ldots, m$. Hence
$\bfm$ must have a variable of each of the $m$ colors.

b. By \cite{ER98}, the Alexander dual $J$ of $J^\prime$ is Cohen-Macaulay of
codimension $m$. For each color class $\Xv_i = \{x_{i1}, \ldots, x_{in_i}\}$, the
ideal $J$ will contain the monomial which is the product of all these variables.

If we for every color class $i$
divide the quotient ring $k[\Xv_1, \ldots, \Xv_m]/J$ by all the variable
differences $x_{i,j} - x_{i,j-1}$ for $j = 2, \ldots, n_i$, we get a quotient
ring $k[x_1, \ldots, x_m]/I$ where $I$ is Artinian since it contains $x_i^{n_i}$
for each $i$. Hence like $J$ the ideal $I$ is Cohen-Macaulay of codimension $m$.
Then the sequence we divided out by must have been a regular sequence, and so
$J$ is a polarization of $I$.
\end{proof}

\begin{remark}
Considering the $\Xv_i$ as color classes, both the Artinian ideal $I$, the polarization $J$ and its Alexander dual are generated by colored monomials. Such ideals and the associated simplicial
complexes have been considered in various settings, like balanced simplicial
complexes by Stanley \cite{St}, relating to the colorful topological Helly theorem by
Kalai and Meshulam, \cite{KM}, and resolutions of such ideals by the second author
\cite{FlC}.
\end{remark}

In Section \ref{sec:AD} we describe the Alexander dual of any polarization $J$ of a maximal ideal power.
But while the description of the generators of $J$ is rather direct, Sections  \ref{sec:lin}  and 
\ref{sec:degto}, the description the Alexander dual
is more subtle. For instance, it is not obvious from the description that there are actually always 
$\binom{n+m-1}{m}$ generators of the Alexander dual. 

\medskip
In \cite{Nem} A. Nematbakhsh gives a precise description of when an ideal generated
by rainbow monomials has linear resolution. His terminology for rainbow monomials
of $d$ colors is "edge monomials of $d$-partite $d$-uniform clutters". He is able to
give a characterization through a remarkable connection to the article
\cite{Mig} where they give a characterization of when a point set in the 
multiprojective space $({\mathbb P}^1)^n$ is arithmetically Cohen-Macaulay
(meaning that the associated multihomogeneous coordinate ring is a Cohen-Macaulay
ring). The characterization given in \cite{Nem} by translating the one in \cite{Mig} 
is the following. 

\begin{theorem} \label{thm:conjRainbow} Let $I$ be generated by rainbow monomials in $d$ colors. Then $I$ 
has a $d$-linear resolution iff:

\begin{itemize}
    \item[a.] Whenever $m_1$ and $m_2$ are two rainbow monomials in $I$ (i.e. generators of degree $d$) with $\lcm(m_1,m_2)$ of degree $\geq d+2$, there is
    a third distinct rainbow monomial $m_3$ in $I$ dividing this least common multiple.
    \item[b.] Whenever $m_1$ and $m_2$ are two rainbow monomials {\underline{\it not}} in $I$ with $\lcm(m_1,m_2)$ of degree $\geq d+2$, there is
    a third distinct rainbow monomial $m_3$ {\underline {\it not}} in $I$ dividing this least common multiple.
\end{itemize}
\end{theorem}

So this says that a subset $A$ of the product $\Xv_1 \cdot \Xv_2 \cdot \cdots \cdot \Xv_d$ gives
an ideal with $d$-linear resolution iff {\it both} $A$ and its complement are 
in some sense convex.

Fr\"oberg's theorem
 \cite{Fr90} characterizes when a monomial ideal generated in degree two has linear resolution. It is easily seen that both this theorem and the theorem above give the following criterion when we have
 rainbow monomials with two colors:
 If $x_ay_b$ and $x_{a^\prime}y_{b^\prime}$ are in $I$, then
 either $x_{a}y_{b^\prime}$ or $x_{a^\prime}y_b$ is in $I$.
Many attempts have been done to generalize Fr\"oberg's theorem to higher
degrees, but none fully successful. For rainbow monomials, however, the above gives such
a generalization.

\begin{example}
Let $X = \{x_1, x_2\}, Y = \{y_1, y_2 \}$ and $Z = \{z_1, z_2\}$ be three color classes
and let $I$ be the ideal generated by the six monomials
\[ x_1y_1z_2, x_1y_2z_1, x_2y_1z_1, x_2y_2z_1, x_2y_1z_2, x_1y_2z_2. \]
Then $I$ does not have linear resolution since $x_1y_1z_1$ and $x_2y_2z_2$ are {\it not} in
$I$ and their least common multiple is not divided by any distinct rainbow monomial {\it not} in $I$.

If we remove $x_1y_1z_2$ we also do not have linear resolution, since now the least
common multiple of $x_2y_1z_2$ and $x_1y_2z_2$ is not divided by any distinct rainbow monomial in $I$. If we remove {\it both} $x_1y_1z_2$ and $x_1y_2z_2$ then we do get a monomial
ideal with linear resolution.
\end{example}

\begin{problem}
 Consider an ideal $J^\prime$ generated by rainbow monomials. Is there a direct criterion on
 this ideal to tell if its Alexander dual is a polarization of a power
 of a graded maximal ideal?
 
 Since an Artinian monomial ideal has linear resolution iff it is a power of
 the graded maximal ideal, this is the same as asking for a criterion for
 the ideal $J^\prime$ to be bi-Cohen-Macaulay: The ideal
 has both linear resolution and is Cohen-Macaulay.
 \end{problem}
 
 Such a description would maybe involve reconstructing the isotone maps $X_i$ given
 in \eqref{eq:sep-iso}, from $J^\prime$. 
 In \cite{Ma15} a cellular resolution
 is computed when $J^\prime$ is the Alexander dual of the {\it standard 
 polarization} $J$ of any Artinian monomial ideal. 
 
 \begin{remark} \label{rem:conj-binary}
 If we for each color class $i$ have only two variables
 $\{x_{i0}, x_{i1}\}$, then a rainbow monomial of degree $m$ may
 be identified with a binary string, say if  $m=6$ then
 $101011$ corresponds to $x_{11}x_{20}x_{31}x_{40}x_{51}x_{61}$. Thus investigating homological properties of ideals generated by rainbow monomials with two variables of each color, corresponds to investigating algebraic and topological properties of sets of binary words.
 \end{remark}
 
 \subsection{Deformations of polarizations}
 In \cite{FN18} the second author and A.Nematbakhsh showed that the letterplace ideals
 $L(2,P)$ (which are polarizations of quadratic Artinian monomial ideals) have unobstructed deformations when the Hasse diagram is a tree. Moreover we computed the full
 deformation family of these ideals. Together with G.Scattareggia we have
 also verified that all deformations of various polarizations of quadratic powers  
 $(x_1,x_2, \ldots, x_m)^2$ lift to global deformations for $m = 3,4$ by computing
 a full global family.
 We have also verified this for the letterplace ideal $L(3,5)$ (introduced in \cite{FGH17}), a cubic ideal.

 \begin{conjecture} \label{conj:conj-defo} Every first order deformation of a polarization of an Artinian monomial ideal (regardless of whether the deformation
   is homogeneous for the standard grading) lifts to a global deformation, with
   any grading on the first order deformation respected by the global deformation. In particular, whenever such a polarization is on a (multigraded) Hilbert scheme, it is a smooth point. \end{conjecture} 
 
If the homogeneous ideal $I$ in the polynomial ring $S$ corresponds to a point on a Hilbert scheme of subschemes of a projective space, then the tangent space of the Hilbert scheme at this point is the degree zero part
$\Hom_S(I,S/I)_0$. By explicit computation with Macaulay 2, one can verify that the
dimension of this space varies quite much for different polarizations. For instance,
for the polarizations of $(x_1,x_2,x_3)^3$, the lowest dimension of the tangent space
occurs for the standard polarization, with dimension $69$.
The largest dimension occurs
for the b-polarization, giving a dimension of $102$ (this ideal is a smooth point on the Hilbert scheme component of ideals of maximal minors of 
$3 \times 5$ matrices of linear forms). Furthermore there are also many values of the dimension of this tangent space between
$69$ and $102$. Thus if they are all smooth points, they would be on many different
components of the Hilbert scheme.

\section{Polarizations and linear syzygy edges}
\label{sec:lin}

We here give our main result, the complete combinatorial description, 
Theorem \ref{thm:LS-XD},
of all polarizations of powers of maximal ideals $(x_1, \ldots, x_m)^n$.
Write $[m] = \{1,2, \ldots, m\}$.

\subsection{Statement and examples}
The essential objects are the rank-preserving isotone maps 
$X_i : \Del_m(n) \pil B(n)$ for $i = 1, \ldots, m$ and conditions on them.
For each $\bfb \in \Del_m(n)$ we get a monomial
$m_i(\bfb) = \prod_{j \in X_i(\bfb)} x_{ij}$ in the variables
$\Xv_i = \{ x_{i1}, \ldots, x_{in} \}$. 
To the vertex $\bfb$ we associate the monomial 
\[ m(\bfb) = \prod_{i=1}^m m_i(\bfb) \] 
and let $J$ be the ideal in $k[\Xv_1, \ldots, \Xv_m]$ generated by the $m(\bfb)$.

If $B$ is a subset of $[m]$, denote by $\ben_B$ the
$m$-tuple $\sum_{i \in B} e_i$. For instance, if $B = [m]$, then
$\ben_B = (1,1,\ldots, 1)$. 

\begin{lemma} \label{lem:lin-felles}
Let $\bfc \in \Delta_m(n+1)$ have support $C \sus \{1,2,\ldots,m\}$. The monomials associated to the vertices in the down-graph $D(\bfc)$ have a common
factor of degree $\bfc - \ben_C$.
This common factor is $\prod_{i \in C} m_i(\bfc - e_i)$.
\end{lemma}

\begin{proof} Fix an element $j \in C$. 
For the order $\geq_k$ we have $\bfc - e_j \geq_k \bfc - e_k$ for every
$k \in C$. Hence $X_k(\bfc - e_k)$ is contained in $X_k(\bfc - e_j)$ for
every $k \in C$. Thus $m(\bfc - e_j)$ has $m_k(\bfc - e_k)$ as a factor
for each $k \in C$.
\end{proof}

\begin{example}
Let $m = 3$ and $\bfc = (c_1,c_2,c_3)$ be in 
$\Delta_3^+(n+1)$. At the left in Figure \ref{fig:LS-Dc}
we illustrate the down triangle  $D(\bfc)$.
\begin{figure}
\begin{tikzpicture}
\draw (-1.5,0)--(1.5,0) node[anchor=east] at (-1.5,0){$(c_1,c_2,c_3-1)$}
node[anchor=south] at (0,0){$(\bfc;2,3)$};
\draw (-1.5,0)--(0,-2.4) node[anchor=south] at (2.3,0)  {$(c_1,c_2-1,c_3)$}
node[anchor=east] at (-0.75,-1.2){$(\bfc;1,3)$};
\draw (1.5,0)--(0,-2.4) node[anchor=north] at (0,-2.4)  {$(c_1-1,c_2,c_3)$}
node[anchor=west] at (0.75,-1.2){$(\bfc;1,2)$};
\filldraw (-1.5,0) circle (2pt);
\filldraw (1.5,0) circle (2pt);
\filldraw (0,-2.4) circle (2pt);
\draw (5,0)--(8,0) node[anchor=south] at (5,0){$\bfn \cdot x_{1i_3}x_{2j_3}$};
\draw (5,0)--(6.5,-2.4) node[anchor=west] at (8,0)  {$\bfn \cdot x_{1i_2} x_{3j_2}$};
\draw (8,0)--(6.5,-2.4) node[anchor=north] at (6.5,-2.4)  {$\bfn \cdot x_{2i_1}x_{3j_1}$};
\filldraw[black] (5,0) circle (2pt);
\filldraw[black] (8,0) circle (2pt);
\filldraw[black] (6.5,-2.4) circle (2pt);
\end{tikzpicture}
  \caption{}
    \label{fig:LS-Dc}
\end{figure}
Let
\[ \bfn = m_1(\bfc-e_1)\cdot m_2(\bfc-e_2) \cdot m_3(\bfc - e_3). \]
Then the monomials associated to the vertices of this down-triangle
are shown to the right in Figure
\ref{fig:LS-Dc}.
\end{example}


\begin{definition}
  An edge $(\bfc;i,j)$ in $\Delta_m(n)$ (where $\bfc \in \Del_m(n+1)$)
is a {\it linear syzygy edge} (LS-edge) if there is a monomial $\bfm$
of degree $(n-1)$ such that 
\begin{equation} \label{eq:LS-ls} m(\bfc - e_i) = x_{jr} \cdot \bfm, \quad
m(\bfc - e_j) = x_{is} \cdot \bfm 
\end{equation}
for suitable variables $x_{jr} \in \Xv_j$ and $x_{is} \in \Xv_i$.
So this edge gives a linear syzygy between the monomials $m(\bfc-e_i)$
and $m(\bfc-e_j)$. 
By Lemma \ref{lem:lin-felles} 
both $m_i(\bfc - e_i)$ and $m_j(\bfc - e_j)$ are common factors
of $m(\bfc-e_i)$ and $m(\bfc-e_j)$. So \eqref{eq:LS-ls} is equivalent to 
\[ m_p(\bfc - e_i) = m_p(\bfc - e_j)  \text{ for every } p \neq i,j,\]
or formulated in terms of the isotone maps
\begin{equation} \label{eq:LS-Xp}  X_p(\bfc-e_i)  = X_p(\bfc - e_j) \text{ for every } p \neq i,j.
\end{equation}
We denote by $\LS(\bfc)$ the set of linear
syzygy edges in the complete down-graph $D(\bfc)$.
\end{definition}

\begin{example} \label{eks:LS-QS}
In Figure \ref{fig:LS-Dc} the edge $(\bfc;2,3)$ is a linear
syzygy edge when $x_{1i_3} = x_{1i_2}$. 
Similarly the edge $(\bfc;1,3)$ is a linear syzygy edge
when $x_{2j_3} = x_{2i_1}$. In general, when the support of $\bfc$ 
has cardinality three and 
when $(\bfc;i,j)$ is {\it not} a linear syzygy edge we call it
a {\it quadratic syzygy (QS) edge}, since we then have a 
quadatic syzygy between the monomials $m(\bfc - e_i)$ and $m(\bfc - e_j)$
\end{example}

\medskip
Here is the main theorem of this article.

\begin{theorem} \label{thm:LS-XD}
The isotone maps $X_1, \ldots, X_m$  determine a polarization
of the ideal $(x_1, \ldots, x_m)^n$ if and only if for every
$\bfc \in \Delta_m(n+1)$, the linear syzygy edges $\LS(\bfc)$ contain
a spanning tree for the down-graph $D(\bfc)$.
\end{theorem}

We prove this towards the end of this section.

\begin{remark} \label{rem:lin-dim} For every $\bfc \in \Delta_m(n+1)$ which has
  support $\{i,j\}$ of cardinality $2$, it is automatic by the
  condition that the $X_p$ are rank-preserving and isotone that
  the edge $(\bfc;i,j)$ is a linear syzygy edge in $\LS(\bfc)$.
  (Condition \eqref{eq:LS-Xp} is empty.)
  So the conditions in Theorem \ref{thm:LS-XD} is automatically
  fulfilled for the $\bfc$ with support of cardinality $2$.
\end{remark}

\begin{corollary} \label{cor:LS-m3} 
When $m = 3$ the isotone maps $X_1, X_2, X_3$ give a polarization
of $(x_1,x_2,x_3)^n$ iff each down-triangle contains at most one QS-edge.
\end{corollary}

\begin{example} \label{eks:LS-m3}
Write $x,y,z$ for $x_1,x_2,x_3$.
Consider the ideal
\begin{equation*}
J = (x_1 x_2 x_3, x_1 x_2 y_2, x_1 x_2 z_2, x_2 y_1 y_2, x_1 y_1 z_2, x_1 z_1 z_2, y_1 y_2 y_3, y_1 y_2 z_2, y_1 z_2 z_3, z_1 z_2 z_3)
\end{equation*}
which is a polarization of $(x,y,z)^3\subset k[x,y,z]$. In Figure \ref{fig:LS-xyz} we denote quadratic syzygy edges by dashed lines.
We see 
that each of the three down-triangles has exactly one quadratic syzygy edge. 
\end{example}

\begin{figure}
\begin{tikzpicture}[]

\draw  (0, 4.33) -- (-2.5,0)--(2.5,0)--(0,4.33) ;

\draw[dashed]  (-1.68, 1.44)--(0,1.44);

\draw (0,1.44)--(1.68, 1.44) ;

\draw[dashed] (-0.85,2.88)--(0,1.44);

\draw (0,1.44)--(0.85,2.88)--(-0.85,2.88) ;

\draw (-1.68,1.44)--(-0.84,0)--(0,1.44);

\draw (0,1.44)--(0.84,0);

\draw[dashed] (0.84,0)--(1.68,1.44);

\draw (-1.5,2.93) node {$x_1x_2y_2$};
\draw (1.5, 2.93) node {$x_1x_2z_2$};
\draw (0.05,4.6) node   {$x_1x_2x_3$};
\draw (-3.3,0) node   {$y_1y_2y_3$};
\draw (3.3,0) node   {$z_1z_2z_3$};
\draw (2.4,1.5) node   {$x_1 z_1z_2$};
\draw (-2.4,1.5) node   {$x_2y_1y_2$};
\draw (-0.84,-0.3) node {$y_1y_2z_2$};
\draw (0.84, -0.3) node {$y_1z_2z_3$};
\draw (0.75, 1.25) node {$x_1y_1z_2$};

\end{tikzpicture}
\caption{}
\label{fig:LS-xyz}
\end{figure}

\begin{proof}[Proof of Corollary \ref{cor:LS-m3}]
A $\bfc \in \Delta_3(n+1)$ has support
of cardinality $1,2$ or $3$. If it is $3$, we have a down-triangle, and
so at least two of the three edges must be linear syzygy edges. Equivalently
at most one of the edges is a QS-edge. 
If the cardinality is $2$, the edge is on the boundary of the simplex induced
by $\Delta_3(n)$ and it is
a linear syzygy edge by the above Remark \ref{rem:lin-dim}.
\end{proof}

\subsection{Linear syzygy paths}
Let $R \sus [m]$ and $\bfc \in \Delta_m(n+1)$ with $R$ contained
in the support of $\bfc$. Let $r,s \in R$. We say $(\bfc;r,s)$ is an {\it $R$-linear
  syzygy edge} if
\[ X_p(\bfc-e_r) = X_p(\bfc-e_s) \text{ for } p \in R\setminus \{r,s \}.
\]
Take note that by isotonicity of the $X_p$, for $p = r,s$:
\[ X_r(\bfc - e_r) \sus X_r(\bfc - e_s), \quad
  X_s(\bfc - e_s) \sus X_s(\bfc - e_r). \]
Let $D_R(\bfc)$ be the complete graph with edges $(\bfc;r,s)$ for
$r,s \in R$.

\begin{lemma} \label{lem:lin-R} Let $\bfc \in \Delta_m(n+1)$.
  If the set of linear syzygy edges in $\LS(\bfc)$ contains a spanning tree for
  $D(\bfc)$, then for each $R \sus \Supp \, (\bfc)$, the set of $R$-linear syzygy edges
  contains a spanning tree for $D_R(\bfc)$.
\end{lemma}

\begin{example}
Consider the case of four variables and $\bfc = (1,1,1,1)$. Write $x,y,z,w$ for $x_1,x_2,x_3,x_4$.
On the left of Figure \ref{fig:LS-R} is the down-graph $D(\bfc)$ with the three thick edges the linear syzygy edges.

Let $R = \{2,3,4\}$. On the right is the
down-graph $D_R(\bfc)$ where the two thick edges are
the $R$-linear syzygy edges and the relevant variables
marked in bold.

\begin{figure}
\begin{tikzpicture}[]
\draw(0,0)--(4,0)--(2.5,1.4);
\draw[very thick] (2.5,1.4)--(0,0)--(1.8,-1.8)--(4,0) ;
\draw[dashed] (1.8,-1.8)--(2.5,1.4);

\draw node[anchor=east] at (-0.1,0) {$x_1y_1w_1$};
\draw node[anchor=south] at (2.5,1.4) {$x_1z_1w_1$};
\draw node[anchor=west] at (4.1,0) {$x_2y_1z_2$};
\draw node[anchor=east] at (1.8,-1.8) {$y_1z_2w_1$};

\draw[very thick] (9.5,1.4)--(7,0)--(11,0);
\draw (11,0)--(9.5,1.4);

\draw node[anchor=north] at (7,0) {$x_1{\mathbf {y_1w_1}}$};
\draw node[anchor=south] at (9.5,1.4) {$x_1{\mathbf{z_1w_1}}$};
\draw node[anchor=north] at (11,0) {$x_2 {\mathbf{y_1z_2}}$};

\end{tikzpicture}
\caption{}
\label{fig:LS-R}
\end{figure}
\end{example}

\begin{proof}
Let $Q$ be the complement of $R$ in $\Supp \, (\bfc)$.
  Let $r$ and $s$ be two elements in $R$. There is a path from $\bfc - e_r$
  to $\bfc - e_s$ in $D(\bfc)$ consisting of linear syzygy edges.
  It may be broken up
  into smaller paths: From $\bfc - e_r = \bfc- e_{r_0}$ to $\bfc-e_{r_1}$,
  from $\bfc - e_{r_1}$ to $\bfc - e_{r_2}$, ...,  from
  $\bfc- e_{r_{p-1}}$ to $\bfc - e_{r_p} = \bfc - e_s$
  where on the path from $\bfc - e_{r_{i-1}}$ to $\bfc - e_{r_i}$
  the only vertices  $\bfc -e_q$ with  $q \in R$ are the end vertices
  $q = r_{i-1}$ and $q = r_i$ while the 
  in between vertices $\cc - e_q$ all have $q \in Q$.
  We claim that each edge from $\bfc - e_{r_{i-1}}$ to $\bfc - e_{r_i}$ is an
  $R$-linear syzygy
  edge. This will prove the lemma.

  Let the path from $\bfc - r_{i-1}$ to $\bfc - r_i$ be
  \[ \bfc - e_{r_{i-1}} = \bfc - e_{q_0}, \bfc - e_{q_1}, \ldots, \bfc - e_{q_t}
    = \bfc - e_{r_{i}} \]
  where $q_1, \ldots, q_{t-1}$ are all in $Q$. We must show that
  \begin{equation} \label{eq:lin-Rlin1}
    X_p(\bfc - e_{r_{i-1}}) = X_p(\bfc - e_{r_i}) \text{ for }
    p \in R\setminus \{ r_{i-1}, r_i \}.
    \end{equation}
  But since the edges on the path are linear syzygy edges we have
  \[ X_p(\bfc - e_{q_{j-1}}) = X_p(\bfc - e_{q_j}) \text{ for }
    p \in \supp(\bfc) \setminus \{ q_{j-1}, q_j\}.\]
 Since $q_1, \ldots, q_{t-1}$ are not in $R$ we get \eqref{eq:lin-Rlin1}
  \end{proof}

Given two $m$-tuples $\bfa = (a_1, \ldots, a_m)$ and
$\bfb = (b_1, \ldots, b_m)$ in $\Delta_m(n)$.
Let $[m] = A \cup B$ be the disjoint set partition
such that 
$a_i \geq b_i$ for $i \in A$ and 
$a_i < b_i$ for $i \in B$.
We let
\begin{equation} \label{eq:LinsyzDist}
d(\bfa, \bfb) = \sum_{i \in B} (b_i - a_i) = \sum_{i \in A} (a_i - b_i)
\end{equation}
be a measure for the distance between $\bfa$ and $\bfb$.
Note that the distance may be measured using only the index set $B$ which in
turn depends on the ordered set $(\bfa, \bfb)$. It 
should thus really be written $B(\bfa, \bfb)$. When we measure the distance
between two vertices, the first will normally be denoted by a variation on
$\bfa$ and the second a variation on $\bfb$. 

We have the partial order $\bfa \leq \bfb$ if each $a_i \leq b_i$.
The least upper bound for $\bfa$ and $\bfb$ in this partial order is
\[ \bfa \vee \bfb = (\max\{a_1, b_1\}, \ldots, \max \{ a_m,b_m\} ). \]

Recall that $m(\bfa)$ and $m(\bfb)$ are the monomials
in positions $\bfa$ and $\bfb$ respectively. The following
is the main and crucial ingredient on the proof of Theorem 
\ref{thm:LS-XD}.

\begin{proposition} \label{pro:LinsyzPath} Given $\bfa, \bfb \in \Delta_m(n)$.
  Suppose for every $\bfc \in \Delta_m(n+1)$ the linear syzygy edges
  $\LS(\bfc)$ contains a spanning tree for the down-graph $D(\bfc)$.
  Then there is a path
  \[ \bfa = \bfb_0, \bfb_1, \ldots, \bfb_N = \bfb \]
  in $\Delta_m(n)$ such that:
  \begin{enumerate}
  \item Every  $\bfb_i \leq \bfa \vee \bfb $,
  \item Every $m(\bfb_i)$ divides the least common multiple
    $\lcm(m(\bfa), m(\bfb))$,
  \item The edge from $\bfb_{i-1}$ to $\bfb_i$ is a linear syzygy
    edge for each $i$.
  \end{enumerate}
  We call such a path an $LS$-path from $\bfa$ to $\bfb$.
\end{proposition}

We first show this when the distance between $\bfa$ and $\bfb$ is one.

\begin{lemma} When the distance between $\bfa$ and $\bfb$ is one, there
is an LS-path from $\bfa$ to $\bfb$.
\end{lemma}

\begin{proof}
  In this case there is a unique $\bfc \in \Delta_m(n+1)$ such that
$\bfa = \bfc - e_i$ and $\bfb = \bfc - e_j$, and then $\bfa \vee \bfb = \bfc$.
Let $T$ in the linear syzygy edges $\LS(\bfc)$ be a spanning tree
for $D(\bfc)$. Then there is a unique path from $\bfa$
  to $\bfb$ in $T$. We show that for any $m(\bfu)$ on this path, 
  $m(\bfu)$ divides $\lcm (m(\bfa), m(\bfb))$. It is enough to show
  for any $k \in \Supp \, (\bfc)$ that any $x_k$-variable in
  $m(\bfu)$ is contained in either the $x_k$-variables of $m(\bfa)$
  or the $x_k$-variables of $m(\bfb)$.

  \medskip
  \textit{Case 1}. If the path from $\bfa$ to $\bfb$ does not contain $\bfc - e_k$, then since
  all edges on the path are linear syzygy edges,
  $X_k(\bfb_{i-1}) = X_k(\bfb_i)$ for every $i$.

  \medskip
  \textit{Case 2}. It the path from $\bfa$ to $\bfb$ contains $\bfc - e_k$, say this is
  $\bfb_t$, then:
  \begin{itemize}
  \item   $X_k(\bfb_{i-1}) = X_k(\bfb_i)$ for $i < t$ and these are
    all equal to $X_k(\bfa)$,
\item $X_k(\bfb_{i}) = X_k(\bfb_{i+1})$ for $i > t$ and these are
  all equal to $X_k(\bfb)$,
\item $X_k(\bfb_t) = X_k(\bfc - e_k)$ is contained in both
  $X_k(\bfa)$ and $X_k(\bfb)$ by the isotonicity of $X_k$.
\end{itemize}
\end{proof}

  \begin{proof}[Proof of Proposition \ref{pro:LinsyzPath}.]
  We do this in three parts.
  
  \noindent{\bf Part A.} In this part we define the setting.
  Take distinct $\bfa$ and $\bfb$ in $\Delta_m(n)$. Assume the
  distance $d(\bfa,\bfb) \geq 2$ since the case of distance $1$ is done above. Let
  \[ B = B(\bfa, \bfb) = \{ i \, | \, b_i > a_i \}, \quad
    A_> = \{ i \, | \, a_i > b_i\},  \quad
    A_= = \{ i \, | \, a_i = b_i\}, \]
  and $A = A_> \cup A_=$. We want to consider $\bfb^\prime$ which are
  in some sense ``close'' to $\bfb$. Let $P(\bfb)$ consist of all
  $\bfb^\prime \in \Delta_m(n)$ such that
  \begin{itemize}
  \item[1.] \begin{itemize}
     \item For $i$ in $B$, $\bfb$ and $\bfb^\prime$ have equal $i$'th coordinate.
     \item For $i$ in $A$, $b_i^\prime \leq a_i$.
       \end{itemize}
  \item[2.] There is some LS-path from $\bfb^\prime$ to $\bfb$ where the vertices
  $\bfu$ on the path satisfy
    \begin{itemize}
\item $\bfu \leq \bfa \vee \bfb$ (which since we are assuming an LS-path, follows from 1 above),
\item $m(\bfu)$ divides $\lcm (m(\bfa), m(\bfb))$.
\end{itemize}
\end{itemize}

Now let the subset $A_1$ of $A$ consist of all coordinate indices $i$ in $A$
such that there is some $\bfb^\prime$ in $P(\bfb)$ with strict
inequality $b^\prime_i < a_i$. Let $A_0$ be the complement $A \setminus A_1$. It is the intersection of all the
$A_=$ associated to $\bfb^\prime$ in $P(\bfb)$.
In particular note i) that $A_1 \supseteq A_>$ (since $\bfb \in P(\bfb)$). So $A_1$ is not empty and $A_0 \sus A_=$.
\begin{equation} \label{eq:LS-PB}
ii)\,\,  b_i^\prime = a_i = b_i \text{ for } i \in A_0, \quad iii) \,\,
d(\bfa, \bfb^\prime) = d(\bfa, \bfb) \text{ for } \bfb^\prime \in P(\bfb),
\end{equation}
the latter because the $B$-sets $B(\bfa,\bfb^\prime) = B(\bfa,\bfb)$
and the distance may be measured by this, confer \eqref{eq:LinsyzDist}.

\medskip
Choose $\beta \in B$ and let $R = A_1 \cup \{ \beta \}$. 
Consider the down-graph $D_R(\bfa + e_{\beta})$. With $\beta$ fixed,
there is by Lemma \ref{lem:lin-R}
an $R$-linear syzygy edge $(\bfa + e_\beta;\beta,\alpha)$ for some
$\alpha$ in $A_1$. This is an edge from
$\bfa$ to $\bfa + e_\beta - e_\alpha$.
Since $\alpha \in A_1$ there is a $\bfb^\prime  \in P(\bfb)$
with $b^\prime_\alpha < a_\alpha$.
Then the $B$-sets (see \eqref{eq:LinsyzDist}) $B(\bfa + e_\beta-e_\alpha,\bfb^\prime) \sus B(\bfa, \bfb^\prime)$, with 
equality unless $a_\beta + 1 = b_\beta$ in which case the the former set comes from removing $\beta$ from the latter. In any case
the distances
\[ d(\bfa + e_\beta - e_\alpha, \bfb^\prime) =  d(\bfa, \bfb^\prime) -1 \, (=
  d(\bfa, \bfb) -1). \]
By induction on distance there is an LS-path
\begin{equation} \label{eq:LinsyzPath}
\bfa + e_\beta - e_\alpha = \bfb^0, \bfb^1, \cdots, \bfb^N= \bfb^\prime.
\end{equation}
So we have 
\begin{itemize}
\item each $m(\bfb^j)$ divides $\lcm(m(\bfa + e_\beta - e_\alpha),
  m(\bfb^\prime))$,
\item each $\bfb^j \leq (\bfa + e_\beta- e_\alpha) \vee \bfb^\prime \leq
    \bfa \vee \bfb$.
\end{itemize}
There may be elements on this LS-path such that $m(\bfb^j)$ does not divide
$\lcm(m(\bfa),m(\bfb))$.
But $m(\bfb^j)$ will divide if we get sufficiently close to $\bfb^\prime$ as
we show in Fact \ref{fact:KKdivab} below. 
If $\bfa + e_\beta - e_\alpha$ and $\bfb^\prime$ have equal $i$'th coordinates for
every $i \in B$, the distance $d(\bfa, \bfb)$ would be $1$, but we are assuming
the distance is $\geq 2$. So $\bfa + e_\beta - e_\alpha$ and $\bfb^\prime$
do not have equal $i$'th coordinate for every $i \in B$. 
Let $\bfb^p$ be the last element on the path \eqref{eq:LinsyzPath} for which
$b^p_{k} \neq b_{k}^{\prime}$ for some $k  \in A_0 \cup B$.

\medskip
\noindent {\bf Part B.} In this part,
in Facts \ref{fact:KKbpk}, \ref{fact:KKdivab}, and
\ref{fact:KKk}, we investigate in detail the path from $\bfb^p$ to
$\bfb^\prime = \bfb^N$. 

\begin{fact} \label{fact:KKbpk} There is a unique $k \in A_0 \cup B$ such that
$b^p_k \neq b_k^\prime$. For every $i \in A_0 \cup B$ and $j = p, \ldots, N$
we have $b^j_i = b^\prime_i = b_i$, save for $j = p$ and $i = k$ when
$b^p_k = b^{p+1}_k -1$ which is $b_k^\prime - 1 = b_k -1$.
\end{fact}

\begin{proof}
  Clearly, by the definition of $\bfb^p$, we have that the $b^j_i$
  are all equal to $b^\prime_i$ for $i \in A_0 \cup B$ and
  $j = p+1, \ldots, N$. 
 If $i \in B$ then $b^\prime_i = b_i$ since $\bfb^\prime \in P(\bfb)$. 
 If $i \in A_0$ then $b^\prime_i = b_i$ by \eqref{eq:LS-PB}.
Furthermore we must have $b^p_k = b^{p+1}_k \pm 1$
  which is $b_k^\prime \pm 1$. Note that
  \begin{equation} \label{eq:KKleq}
    b_i^p \leq \max \{a_i,b_i\} = b_i, \quad i \in A_0 \cup B.
  \end{equation}
Since the edge from $\bfb^p$ to $\bfb^{p+1}$ is an LS-edge, there are exactly 
two coordinates $k, \ell$ where $\bfb^p$ and $\bfb^{p+1}$ are distinct. 
Since $b^{p+1}_k = b_k$ we must by \eqref{eq:KKleq} have $b^p_k = b^{p+1}_k -1$.
Then we will have $b^p_\ell = b^{p+1}_\ell + 1$. If $\ell \in A_0 \cup B$ then 
$b^{p+1}_\ell = b^\prime_\ell = b_\ell$ which together with \eqref{eq:KKleq}
gives a contradiction.
Thus we have a unique $k$ in $A_0 \cup B$. Whence when $i \in A_0 \cup B$ and $i \neq k$
we have $b^p_i = b^{p+1}_i = b^\prime_i = b_i$.
\end{proof}

\begin{fact} \label{fact:KKdivab} For all $\bfb^j$ with $j = p, \ldots, N$
  we have

  i) $\bfb^j \leq \bfa \vee \bfb$,

  ii) $m(\bfb^j)$ divides $\lcm(m(\bfa),m(\bfb))$.
\end{fact}

\begin{proof} Part i) is already noted when we defined the path from
  $\bfa + e_\beta - e_\alpha$ to $\bfb^\prime$. Now we do part ii).
  We know that $m_t(\bfb^j)$ divides $\lcm(m_t(\bfa + e_\beta - e_\alpha),
  m_t(\bfb^\prime))$ for every $t$.

  a. Let $t \in A_1$. There is an $R = A_1 \cup \{\beta \}$-linear syzygy
  between $m(\bfa)$ and $m(\bfa + e_\beta - e_\alpha)$ and so
  $m_t(\bfa) = m_t(\bfa + e_\beta - e_\alpha)$ for $t \in A_1 \setminus
  \{ \alpha \}$. For $t = \alpha$ then $m_\alpha(\bfa + e_\beta- e_\alpha)$
  divides $m_\alpha(\bfa)$ by $X_\alpha$ being isotone.
  From this and the defining requirements on $\bfb^\prime$,
  it follows that $m_t(\bfb^j)$ divides $\lcm(m(\bfa),m(\bfb))$ for $t \in A_1$.

  b. Let now $t \in A_0 \cup B$. The edges on the path between
  $\bfb^p$ and $\bfb^\prime$ are LS-edges. It follows then
  from Fact \ref{fact:KKbpk}
  that for each 
  $t \in (A_0 \cup B) \setminus \{ k \}$ and $j = p, \ldots, N$ that
  $m_t(\bfb^j) = m_t(\bfb^\prime)$.
  When $t = k$, since $b^p_k = b_k^{p+1} - 1$ the part $m_k(\bfb^p)$ divides $m_k(\bfb^{p+1})$ and we will further have
  all $m_k(\bfb^j)$ equal for $j = p+1, \ldots, N$, since these
  $\bfb^j$ are related by LS-edges, and have the same $k$'th
  coordinate. The upshot is that also $m_k(\bfb^j)$
  divides $m_k(\bfb^\prime)$.
  Thus $m_t(\bfb^j)$ divides $m_t(\bfb^\prime)$ for every $t \in A_0 \cup B$.
  Since $\bfb^\prime \in P(\bfb)$, the $m_t(\bfb^\prime)$ divide
$\lcm(m(\bfa),m(\bfb))$ and we are done.
\end{proof}

The following is the main technical detail that makes
the proof work. It ensures that we can use induction on
distance in Part C. To achieve this we need $A_0$ as
small as possible, and therefore introduced the
neighbourhood $P(\bfb)$ of $\bfb$. (But this had to be balanced against
$A_1$, the complement of $A_0$ in $A$, not being too big in order to
construct the LS-path in \eqref{eq:LinsyzPath} by induction.)

\begin{fact} \label{fact:KKk}
  The coordinate $k \in B$.
\end{fact}

\begin{proof}
  By Fact \ref{fact:KKdivab} $m(\bfb^j)$ divides $\lcm(m(\bfa),m(\bfb))$ and
  $\bfb^j \leq \bfa \vee \bfb$ for $j = p, \ldots, N$.
  If $k \in A_0$ then by Fact \ref{fact:KKbpk} $b^p_i = b^\prime_i = b_i$ for 
  each $i \in B$. So $\bfb^p$ fulfills the requirement to be in $P(\bfb)$.
  But then by \eqref{eq:LS-PB} $b^p_k = a_k = b_k$ contradicting 
  Fact \ref{fact:KKbpk}.
  \end{proof}

\medskip
\noindent {\bf Part C.} In this last part we splice three
paths together to make an LS-path from $\bfa$ to $\bfb$.
Consider the distance between $\bfa$ and $\bfb^p$.
For $i \in A$ we have $b^p_j \leq a_j$ since $\bfb^p \leq \bfa \vee \bfb$.
For $i \in B \setminus \{ k \}$ we have $b^p_i = b_i^\prime = b_i > a_i$
and when $i = k$ then $a_k < b_k$ and $b^p_k = b_k -1$.
Thus, by looking at the terms with coordinates in $B$,
the distance $d(\bfa, \bfb^p) = d(\bfa, \bfb) - 1$.
By induction there is an LS-path from $\bfa$ to $\bfb^p$.
We now have three LS-paths which we splice:

\begin{itemize}
\item The LS-path from $\bfa$ to $\bfb^p$. All $m(\bfu)$ on this path
  have i) $\bfu \leq \bfa \vee \bfb^p \leq \bfa \vee \bfb$
  and ii) $m(\bfu)$ divides $\lcm(m(\bfa), m(\bfb^p))$ which again divides
  $\lcm(m(\bfa),m(\bfb))$ by Fact \ref{fact:KKdivab}.
\item The LS-path from $\bfb^p$ to $\bfb^\prime$. We refer to Fact
  \ref{fact:KKdivab} concerning the terms here.
\item The LS-path from $\bfb^\prime$ to $\bfb$ with properties required 
  by the definition of $P(\bfb)$. 
\end{itemize}

  Splicing these three LS-paths together we get a path of linear syzygy
  edges from $m(\bfa)$ to $m(\bfb)$ where all $m(\bfu)$ on this path
  have i) $\bfu \leq \bfa \vee \bfb$ and ii) $m(\bfu)$ divides
  $\lcm(m(\bfa), m(\bfb))$. Thus we have an LS-path from $\bfa$ to $\bfb$.
  \end{proof}
  
\subsection {Proofs of the main Theorem \ref{thm:LS-XD}}

  We show first Part a., that if the isotone maps $\{X_i \}$ give a polarization,
  then for each $\bfc \in \Delta_m(n+1)$ the linear syzygy edges
  $\LS(\bfc)$ of the down-graph $D(\bfc)$ contain a spanning tree for this down-graph.
  
  \begin{proof}[Proof of Theorem \ref{thm:LS-XD}, Part a.] We assume that
    the isotone maps $\{X_i \}$ give an ideal $J$ which is a polarization.
    We shall prove that every down-graph $D(\bfc)$ contains a
    spanning tree of linear syzygy edges.
  For simplicity we shall assume
  $\Supp \, (\bfc)$ has full support $[m] = \{1,2,\ldots, m\}$.
  The arguments work just as well in the general case.
  Since by Lemma \ref{lem:lin-felles}
  \[ \bfm = \prod_{i = 1}^m  m_i(\bfc - e_i) \] of degree $\bfc - \bfen $ is
  a divisor of $m(\bfc - e_v)$ for any $\bfc - e_v$ in $D(\bfc)$, we may
  write $m(\bfc - e_v) = \bfm \cdot n(\bfc - e_v)$ where $n(\bfc - e_v)$ has
  degree $\bfen - e_v$. For two distinct vertices $\bfc-e_v $ and $\bfc-e_w $ in
  $D(\bfc)$ we define the distance $d(m(\bfc-e_v), m(\bfc-e_w))$ to be the
  number of $k \in [m]$ such that either:
  \begin{itemize}
    \item The (unique) $x_k$-variables of $n(\bfc-e_v)$ and of
      $n(\bfc-e_w)$ are distinct,
    \item $k = v$ (then $n(\bfc-e_v)$ has no $x_v$-variable),
    \item $k = w$ (then $n(\bfc-e_w)$ has no
      $x_w$-variable),
    \end{itemize}

    Note that if the distance between $m(\bfc-e_v)$ and $m(\bfc-e_w)$ is $2$,
    then the set of $k$'s is $\{v,w\}$ and there is a linear syzygy between these monomials.
    Suppose now the vertices of $D(\bfc)$ can be divided into two distinct
    subsets $V_1$ and $V_2$ such that there is no linear syzygy edge between
    a vertex in $V_1$ and a vertex in $V_2$.

    Let $\bfc-e_v $ in $V_1$ and $\bfc -e_w$ in $V_2$ be such that the distance $d$
    between $m(\bfc-e_v)$ and $m(\bfc-e_w)$ is minimal. We must have $d \geq 3$ and
    the number of vertices $m \geq 3$. 
    For simplicity we may assume $v = 1$ and $w = 2$
    and that we may write
    \[ n(\bfc-e_2) = x_{1i_1}x_{3i_3} \cdots x_{mi_m}, \quad
      n(\bfc-e_1) = x_{2j_2}x_{3j_3} \cdots x_{mj_m}, \]
    where $x_{pi_p} \neq x_{pj_p}$ for $p = 3,\ldots, d$ and
    $x_{pi_p} = x_{pj_p}$ for $p > d$ where $d \geq 3$. 

    Consider the graded ring $k[\Xv_1, \ldots, \Xv_m]/J$ and divide out by the
    regular sequence $x_{pi_p} - x_{pj_p}$ for $p = 4, \ldots, d$.
This is a regular sequence, since we assume we have a polarization.
    We get
    a quotient algebra $k[\Xv_1^\prime, \ldots, \Xv_m^\prime]/J^\prime$ and denote by
    $x_p$ the class $\overline{x_{pi_p}} = \overline{x_{pj_p}}$ for
    $p \geq 4$. In $J^\prime$ we have generators
    \begin{align*}
      \overline{m}(\bfc-e_2) = \overline{\bfm} \cdot \overline{n}(\bfc-e_2), \quad
      & \overline{n}(\bfc-e_2) = x_{1i_1}x_{3i_3}x_4 \cdots x_m \\
      \overline{m}(\bfc-e_1) = \overline{\bfm} \cdot \overline{n}(\bfc-e_1), \quad
      & \overline{n}(\bfc-e_1) = x_{2j_2}x_{3j_3}x_4 \cdots x_m
    \end{align*}
        Now $x_{3i_3} - x_{3j_3}$ is a non-zero divisor of
        $k[\Xv_1^\prime, \cdots , \Xv_m^\prime]/J^\prime$.
    Consider
    \[ (x_{3i_3} - x_{3j_3})x_{1i_1}x_{2j_2}x_4 \cdots x_m \cdot \overline{\bfm}. \]
    It is zero in this quotient ring, and so
        \[ \overline{\bfm^\prime} = x_{1i_1}x_{2j_2}x_4 \cdots x_m \cdot
        \overline{\bfm} \]
    is zero in this quotient ring and so must be a generator of
    $J^\prime$ of degree $\bfc - e_3$.
    But then the generator of this degree in the polarization $J$ must be
    \[ \bfm^\prime = x_{1i_1}x_{2j_2}x_{4k_4}\cdots x_{mk_m}  \cdot \bfm \]
    where each $k_p$ is either $i_p$ or $j_p$.
    Hence all $k_p = i_p = j_p$ for $p > d$. But then we see that the
    distance between $\bfm^\prime$ and $m(\bfc-e_2)$ is $\leq d-1$ and similarly
    the distance
    between $\bfm^\prime$ and $m(\bfc-e_1)$ is $\leq d-1$.
    Whether $\bfm^\prime$ is now in $V_1$
    or in $V_2$ we see that this contradicts $d$ being the minimal distance.
  \end{proof}
  
\begin{proof}[Proof of Theorem \ref{thm:LS-XD}, Part b.]

  We shall now prove that if each down-graph $D(\bfc)$ contains a spanning tree of
  linear syzygy edges, then $J$ will be a polarization. 
  Order the variables in each $\Xv_i$ in a sequence $x_{i1}, x_{i2}, \ldots, x_{in}$.
  Let $\Xvp_i$ consist of $x_{i1}, \ldots, x_{ip_i}, x_i$ so we have a surjection
 $\Xv_i \pil \Xvp_i$ for each $i$ sending $x_{ij}$ to itself for $j \leq p_i$,
 and to $x_i$ for $j > p_i$. 
 Denote the image of $J$ in $k[\Xv_1^\prime, \ldots , \Xv_m^\prime]$ by $J^\prime$ and the image of $m(\bfa)$ by
 $m^\prime(\bfa)$.
 The quotient ring $k[\Xvp_1, \ldots, \Xvp_m]/J^\prime$
 is obtained from $k[\Xv_1, \ldots \Xv_m]/J$ by dividing out by variable 
 differences $x_{ij} - x_{i,j+1}$ for $i = 1, \ldots m$ and $j >  p_i$. We 
 assume this is a regular sequence. We show that if we now divide out by 
 $x_{i,p_i}- x_{i,p_i+1}$ this is a non-zero divisor of $k[\Xvp_1, \ldots, \Xvp_m]/J^\prime$.
 By continuing we get eventually that
  $k[x_1, \ldots, x_m]/(x_1, \ldots, x_m)^n$ is a regular quotient of
  $k[\Xv_1, \ldots, \Xv_m]/J$ and so $J$ is a polarization of $(x_1, \ldots, x_m)^n$.

  \medskip
  We write $x_i^\prime = x_{i,p_i}$. 
  Suppose $(x_i^\prime - x_i) \cdot \bff = 0$ in
  $k[\Xvp_1, \ldots, \Xvp_m]/J^\prime $ where $\bff$ is a polynomial.
  By Lemma \ref{lem:sepx01} (with $x_i^\prime = x_{i,p_i}$ taking the place
  of $x_0$), we must show that if $\bfm$ is a monomial such that 
  $x_i^\prime \cdot \bfm = 0$ and $x_i \cdot \bfm = 0$, 
  then $\bfm = 0$ in the quotient ring.
  So some generator
  $m^\prime(\bfa)$ of $J^\prime $ divides $x_i^\prime \bfm $ and some generator
  $m^\prime(\bfb)$ divides $x_i \bfm $.

  By Proposition \ref{pro:LinsyzPath} there is an LS-path from
  $m(\bfa)$ to $m(\bfb)$ consisting of linear syzygy edges and such that
  each $\bfu$ on this path has $\bfu \leq \bfa \vee \bfb$ and $m(\bfu)$ on this 
  path divides $\lcm (m(\bfa), m(\bfb))$. The image
  $m^\prime(\bfu)$ then divides $x_i^\prime x_i \bfm $. 
  We will show by induction on the length of the
  path that some monomial $m^\prime(\bfu)$ on this path divides $\bfm $, and
  so $\bfm $ is zero in the quotient ring
  $k[\Xvp_1, \ldots, \Xvp_m]/J^\prime$.

  \medskip
  If the path has length one, there is a linear syzygy edge between
  $m(\bfa)$ and $m(\bfb)$. Write
  \[ x_i^\prime \cdot \bfm = m^\prime(\bfa) \cdot n^0(\bfa), \quad
    x_i \cdot \bfm = m^\prime(\bfb) \cdot n^0(\bfb). \]
  Write also $\bfm = (x_i^{\prime})^p (x_i)^q \cdot \nn $ where
  $\nn$ does not contain $x_i^\prime$ or $x_i$.
  If none of $m^\prime(\bfa)$ or $m^\prime(\bfb)$ divides $\bfm $, then
  \[ m^\prime(\bfa) = (x_i^\prime)^{p+1} (x_i)^{q^\prime}
    \cdot n^1(\bfa), \quad
     m^\prime(\bfb) = (x_i^\prime)^{p^\prime} (x_i)^{q+1}
     \cdot n^1(\bfb), \]
   where $p^\prime \leq p$ and $q^\prime \leq q$ (and $n^1(\bfa)$ and
   $n^1(\bfb)$ do not contain $x_i^\prime$ or $x_i$).
   But since the edge from
   $\bfa $ to $\bfb $ is a linear syzygy edge, we must have $p^\prime = p$,
   $q^\prime = q$. But a linear syzygy edge involves variables of distinct $x_i$-type, which is
   not so here. Thus one of $m^\prime(\bfa)$ or $m^\prime(\bfb)$ must divide $\bfm$.

   \medskip
   Suppose now the path has length $\geq 2$. 
   
   \medskip
   \noindent {\bf Case $a_i \geq b_i$.}
   Let $\bfa $ to $\bfa^\prime $ be the first edge along the path.
   Then the coordinate $a_i^\prime \leq a_i$. 

   If i) the coordinates $a_i$ and $a_i^\prime $ are equal, the $x_i$-type variables
   of $m(\bfa)$ and $m(\bfa^\prime)$ are the same, since this is a linear
   syzygy edge. Since $m^\prime(\bfa)$ divides $x_i^\prime \bfm$ we get that
   $m_i^\prime(\bfa^\prime)$ divides $x_i^\prime \bfm$. 
   If ii) $a_i^\prime < a_i$ then when going from $m(\bfa) $ to $m(\bfa^\prime)$ some $x_i$-type variable
   drops out from $m_i(\bfa)$ by isotonicity of $X_i$ and so also 
   $m_i^\prime(\bfa^\prime)$ divides $x_i^\prime \bfm$.
   
   Since the path from $\bfa$ to $\bfb$ is an LS-path, $m^\prime(\bfa^\prime)$ divides $x_i^\prime x_i \bfm$.
     For $j \neq i$, then $m_j^\prime(\bfa^\prime)$ must divide $\bfm$ since $m^\prime_j(\bfa^\prime)$ contains no $x_i$-type variable. The upshot is that $m^\prime(\bfa^\prime)$ divides $x_i^\prime \bfm$. Considering the LS-path from $\bfa^\prime$ to 
     $\bfb$, by induction on path length, some $m^\prime(\bfu)$ along this path divides $\bfm $.

\medskip
  \noindent {\bf Case $a_i \leq b_i$.}
   Let $\bfb $ to $\bfb^\prime $ be the first edge along the path going from $\bfb$
   to $\bfa$.
   Then the coordinate $b_i^\prime \leq b_i$. Exactly the same argument as in the case
   above works in this case also.
\end{proof}

\section{Three variables case: Classification of isotone maps} \label{sec:trevar}
In the three variables case there is a particular nice classification of the isotone
rank preserving maps from $\Delta_3(n) \pil B(n)$. 

Let $\Delta_3(n)$ have the 
partial order $\geq_1$.
The symmetric group $S_{n}$ acts on the Boolean poset $B(n)$ and so
acts on the set of isotone rank-preserving maps 
\[ X : \Delta_3(n) \pil B(n). \]
Note that for $\bfc \in \Delta_3^+(n)$, an edge $(\bfc;2,3)$ being a quadratic syzygy edge,
see Example \ref{eks:LS-QS}, is a property invariant under the action of the group $S_n$.
The following completely describes the combinatorics of such isotone maps.
This is quite nice for the three variables case. For the $m$-variables case in general
it seems harder to see a nice description.

\begin{proposition} Let $n \geq 1$.
  The orbits under the action of $S_{n}$ of rank-preserving
  isotone maps $X : \Delta_3(n) \pil
  B(n)$ are in one to one correspondence
  with subsets of $\Delta_3^+(n+1)$. An isotone map $X$ corresponds to those $\bfc \in \Delta_3^+(n+1)$ such that the $(\bfc;2,3)$-edges of $\Delta_3(n)$ are quadratic syzygy edges for $X$. 
  
  Since $\Delta_3^+(n+1)$ is in
  bijection to $\Delta_3(n-2)$ which has cardinality $\binom{n}{2}$, the number of such orbits is $2^{\binom{n}{2}}$.
  \end{proposition}

  \begin{proof} In each orbit there is exactly one isotone map where the
    sets
    \begin{equation*} \label{eq:trevarX0}  X(0,0,n) \sus  X(1,0,n-1) \sus X(2,0,n-2) \sus  \cdots \sus
      X(0,0)
      \end{equation*} 
      are \begin{equation} \label{eq:trevar1n}
          \emptyset \sus \{1\} \sus \{1,2\} \sus \cdots \sus \{1,2,\ldots, n\} 
    \end{equation}
    so we need only consider isotone maps which fulfill this condition and show
    that they are in one-one correspondence with subsets $Q$ of $\Delta_3(n-2)$.

    Given such an isotone map $X$ we get a subset $Q$ of QS-edges, those
    $(c_1,c_2,c_3)$ in $\Delta_3^+(n+1)$ such that $X(c_1,c_2-1,c_3) \neq X(c_1,c_2,c_3-1)$. Conversely given
    a subset $Q$ of $\Delta_3^+(n+1)$ we show that this determines uniquely
    an isotone $X$ fulfilling \eqref{eq:trevar1n} whose associated set of QS-edges is $Q$. The essential idea is the following simple observation:

    \medskip
    Let $A$ be a set of cardinality $p-1$ where $p \geq 1$, $B,C$ sets containing $A$ of
    cardinality $p$, and $D$ a set of cardinality $p+1$ containing $B$ and $C$.
    For fixed $A,B,D$, then either $C = B$ or $C = A \cup (D\setminus B)$.
    In the latter case note that $A = B \cap C$ and $D = B \cup C$. 

    \medskip
    Note now first that $X(0,r,n-r)$ is always the empty set. 
    Let $A,B,C,D$ be the four sets of cardinalities respectively
    $p-1, p, p$ and $p+1$:
    \[ X(p-1,1,n-p), X(p,1,n-p-1), X(p,0,n-p), X(p+1,0,n-p-1). \]
    The last two are known by \eqref{eq:trevar1n}. By induction on $p$ the first is known. Then the second is completely determined by whether the edge
    $((p,1,n-p);2,3)$ is a LS- or QS-edge. Thus given $Q$ we may determine
    all the $X(p,1,n-p-1)$ for $p = 1, \ldots, n-1$. Then we may continue
    and similarly determine all $X(p,2,n-p-2)$ for $p = 1, \ldots, n-2$ and
    so on.
    \end{proof}

\section{Degree two case: Polarizations of $(x_1,x_2, \ldots, x_m)^2$}
\label{sec:degto}
We show that polarizations of the second power
$(x_1,x_2, \ldots, x_m)^2$ of the maximal ideal are in one-to-one
correspondence with oriented trees with edges labels $1,2, \ldots,m$.
Moreover we show that the isomorphism classes of polarizations are
in one-to-one correspondence with trees on $(m+1)$ vertices.

\medskip
First we recall a construction given in \cite{Fl09}. Given a directed
tree $T$ (the edges are directed) with edges labelled $1,2,\ldots,m$.
The label of an edge $e$ is denoted $l(e)$
Let $\Xv_i = \{x_{i0}, x_{i1} \}$ for $i = 1,\ldots,m$. We construct
a monomial ideal $K(T)$ in $k[\Xv_1,\ldots,\Xv_m]$ generated by 
$(m+1)$ monomials, one for each vertex of $T$.
Consider a vertex $v$ of the tree $T$. If an edge $e$ in the tree $T$ is pointing 
in the direction towards $v$ (this does not meant that $e$ and $v$ have to be incident) let $e_{to}(v) = 1$ and if
$e$ is pointing in the opposite  direction, let $e_{to}(v) = 0$.
To the vertex $v$ of the tree $T$ associate a monomial $m_v$ of degree $m$
\[ m_v = \prod_{e \in T} x_{l(e),e_{to}(v)}, \] and let
$I(T)$ be the ideal in $k[\Xv_1,\ldots,\Xv_m]$ generated
by these monomials. We see that the $m_v$ are rainbow monomials. By \cite{Fl09} this is a Cohen-Macaulay
monomial ideal of codimension two  with $m$-linear resolution.

We now construct another monomial ideal $J(T)$ in this polynomial ring
as follows. For each pair of distinct vertices $v,w$ of $T$ there is
a unique path (forgetting the direction of the edges) between
$v$ and $w$. Let $e$ be the edge on this path incident to $v$ and 
$f$ the edge on this path incident to $w$. Define the monomial $m_{v,w}$ of degree two
\[ m_{v,w} = x_{l(e),e_{to}(w)} x_{l(f),f_{to}(v)}, \]
and let $J(T)$ be the ideal generated by all these monomials.

\begin{theorem} \label{thm:degto-T} Let $T$ be a directed tree with edges labelled by $1,2,\ldots, m$. \begin{itemize}
\item[a.] The ideals $J(T)$ and $I(T)$ are Alexander duals.
    \item[b.] The ideal $J(T)$ is a polarization of $(x_1, \ldots, x_m)^2$, and every polarization of the latter ideal is 
    isomorphic as a monomial ideal to some $J(T)$.
    \item[c.] Two polarizations $J(T)$ and $J(T^\prime)$ are isomorphic as
    monomial ideals
    if and only if the underlying (unlabelled, undirected) trees of $T$ and $T^\prime$ are isomorphic.
\end{itemize}
\end{theorem}

\begin{example} \label{ex:degto-fire} The two trees in Figure \ref{fig:degto-fire} give the non-isomorphic polarizations of $(x_1,x_2,x_3)^2$.
\begin{figure}
\begin{tikzpicture}
\draw[directed] (0,0)--(0,1) node[anchor=west] at (0,0.7){$1$};
\draw[directed] (0,0)--(0.75,-0.5) node[anchor=north] at (0.5,-0.33)  {$2$};
\draw[directed] (0,0)--(-0.75,-0.5) node[anchor=south] at (-0.5,-0.33)  {$3$};
\filldraw (0,0) circle (2pt);
\filldraw (0,1) circle (2pt);
\filldraw (0.75,-0.5) circle (2pt);
\filldraw (-0.75,-0.5) circle (2pt);

\draw[directed] (4,0)--(3,0) node[anchor=north] at (3.5,0){$1$};
\draw[directed] (5,0)--(4,0) node[anchor=north] at (4.5,0)  {$2$};
\draw[directed] (6,0)--(5,0) node[anchor=north] at (5.5,0)  {$3$};
\filldraw (3,0) circle (2pt);
\filldraw (4,0) circle (2pt);
\filldraw (5,0) circle (2pt);
\filldraw (6,0) circle (2pt);

\end{tikzpicture}
    \caption{}
    \label{fig:degto-fire}
\end{figure}
The first tree gives the standard polarization
\[ (x_{11}x_{12}, x_{11}x_{21}, x_{21}x_{22}, x_{11}x_{31}, x_{21}x_{31},
    x_{31}x_{32}). \] The second tree gives the
b-polarization
\[ (x_{11}x_{12}, x_{11}x_{22}, x_{21}x_{22}, x_{11}x_{32}, x_{21}x_{32},
    x_{31}x_{32}), \]
    which is the letterplace ideal $L(2,3)$ of \cite{FGH17}.
\end{example}

\begin{example} \label{ex:degto-five}
The trees with five vertices, Figure \ref{fig:degto-five}, decorated with direction and labelling, give the three non-isomorphic polarizations of 
$(x_1,x_2,x_3,x_4)^2$.

\begin{figure}
\begin{tikzpicture}
\draw[directed] (0,0)--(0,1) node[anchor=west] at (0,0.7){$1$};
\draw[directed] (0,0)--(1,0) node[anchor=north] at (0.7,0)  {$2$};
\draw[directed] (0,0)--(0,-1) node[anchor=east] at (0,-0.7)  {$3$};
\draw[directed] (0,0)--(-1,0) node[anchor=south] at (-0.7,0)  {$4$};
\filldraw (0,0) circle (2pt);
\filldraw (0,1) circle (2pt);
\filldraw (1,0) circle (2pt);
\filldraw (0,-1) circle (2pt);
\filldraw (-1,0) circle (2pt);

\draw[directed] (3,0)--(2.5,0.75) node[anchor=west] at (2.67,0.5){$1$};
\draw[directed] (3,0)--(2.5,-0.75) node[anchor=west] at (2.67,-0.5)  {$2$};
\draw[directed] (4,0)--(3,0) node[anchor=north] at (3.5,0)  {$3$};
\draw[directed] (5,0)--(4,0) node[anchor=north] at (4.5,0)  {$4$};
\filldraw (3,0) circle (2pt);
\filldraw (2.5,0.75) circle (2pt);
\filldraw (2.5,-0.75) circle (2pt);
\filldraw (4,0) circle (2pt);
\filldraw (5,0) circle (2pt);

\draw[directed] (7.5,0)--(6.5,0) node[anchor=north] at (7,0){$1$};
\draw[directed] (8.5,0)--(7.5,0) node[anchor=north] at (8,0)  {$2$};
\draw[directed] (9.5,0)--(8.5,0) node[anchor=north] at (9,0)  {$3$};
\draw[directed] (10.5,0)--(9.5,0) node[anchor=north] at (10,0)  {$4$};
\filldraw (6.5,0) circle (2pt);
\filldraw (7.5,0) circle (2pt);
\filldraw (8.5,0) circle (2pt);
\filldraw (9.5,0) circle (2pt);
\filldraw (10.5,0) circle (2pt);
\end{tikzpicture}
    \caption{}
    \label{fig:degto-five}
\end{figure}

The first tree gives the standard polarization. 
The second tree gives the polarization:
\[(x_{11}x_{12},
x_{21}x_{22},
x_{31}x_{32},
x_{41}x_{42},
x_{11}x_{21},
x_{11}x_{32},
x_{11}x_{42},
x_{21}x_{32},
x_{21}x_{42},
x_{31}x_{42})
\]
The third tree gives the b-polarization:
   \[
(x_{11}x_{12},
x_{21}x_{22},
x_{31}x_{32},
x_{41}x_{42},
x_{11}x_{22},
x_{11}x_{32},
x_{11}x_{42},
x_{21}x_{32},
x_{21}x_{42},
x_{31}x_{42}),
\] 
which is the letterplace ideal $L(2,4)$.

\end{example}
\begin{proof}[Proof of Theorem \ref{thm:degto-T}] 
a. Consider the generator $m_{v,w} = x_{l(e),e_{to}(w)} x_{l(f),f_{to}(v)}$ of $J(T)$. 
Recall that $e$ is incident to $v$. 
Let $u$ be another vertex of $T$.
If the first variable of $m_{v,w}$ is not in the monomial $m_u$ then
$w$ and $u$ are in distinct directions from $v$. If the second
variable is not in $m_u$, then $v$ and $u$ are in distinct directions from $w$, but these two situations together are not possible in a tree.

Hence all the $\binom{m+1}{2}$ monomials $m_{v,w}$ are in the Alexander dual of
$I(T)$. But $I(T)$ has $m$-linear resolution and is Cohen-Macaulay of codimension two
by \cite{Fl09}.
Such an ideal is easily seen to have multiplicity 
$\binom{m+1}{2}$.
For a squarefree monomial ideal $I$ in a polynomial ring $S$,  the multiplicity of $I$ (or 
really of $S/I$) equals the number of the facets of maximal dimension of $\Delta(I)$,
see Section 6.11 and equation (6.4) in \cite{HH11}.
Since $I(T)$ is Cohen-Macaulay all of its facets have the same dimension. 
So there are $\binom{m+1}{2}$ facets. Since $I(T)$ has codimension two
the facets have cardinality $2m-2$. The Alexander dual 
of $I(T)$ is then generated by $\binom{m+1}{2}$
quadratic monomials by the standard correspondence between facets of 
a squarefree monomial ideal and the generators of its Alexander dual,
see Lemma 1.5.4 and Corollary 1.5.5 of \cite{HH11}.
So these generators must be precisely the generators of $J(T)$ making it the Alexander dual.

b. The ideal $J(T)$ has codimension $m$, since its Alexander dual $I(T)$ has 
generators of degree $m$, loc.cit.
Dividing out by the variable differences $x_{i0}- x_{i1}$ we easily see that we get the ideal $(x_1, \ldots, x_m)^2$
which also has codimension $m$. Hence this sequence of variable 
differences is a regular sequence for $k[\Xv_1, \ldots, \Xv_m]/J(T)$, and so $J(T)$ is a polarization of 
$(x_1,\ldots,x_m)^2.$

   Conversely let $J$ be a polarization of $(x_1, \ldots, x_m)^2$.
   Then $J$ is Cohen-Macaulay of codimension $m$ with $2$-linear resolution.
   The Alexander dual $I$ will then be Cohen-Macaulay of codimension $2$ with $m$-linear resolution, \cite{ER98}.
   By \cite[Prop.2.4]{Fl09} there is a tree $T$ such that
   that $I$ is the image of the $I(T)$ by a map of 
   polynomial rings $k[\Xv_1, \ldots, \Xv_m]$ (the ring where $I(T)$ lives) to  $k[\Xv_1, \ldots, \Xv_m]$ (the ring where $I$ lives),
   sending the variables in the former ring to monomials in the latter ring. But since the generators of $I$ and $I(T)$ have
   the same degree, we must map the variables to variables. 
   Hence $I$ is isomorphic to $I(T)$ as a monomial ideal.

c. If the underlying trees are the same then clearly $T$ and $T^\prime$ are related by the action of an element of $S_m \ltimes 
(\hele_2)^m$, and so also $J(T)$ and $J(T^\prime)$. If the underlying trees are different, then $I(T)$ 
and $I(T^\prime)$ are not isomorphic: Their linear syzygies are given precisely by the edges of the trees. 
Any isomorphism between
$I(T)$ and $I(T^\prime)$ would induce a bijection of monomial generators,
corresponding to a bijection of vertices of the trees, such that the corresponding linear syzygies between generators, would correspond
to a bijection between the edges of the trees.
\end{proof}

Together with the discussion in Subsection \ref{subsec:conj-ball} the following
shows that polarizations of $(x_1, \ldots, x_m)^2$ define simplicial balls.
For the notion of linear quotients consult \cite[Section 8.2.1]{HH11}.

\begin{proposition} The ideal $I(T)$ has linear quotients, and so equivalently
(\cite[Prop.8.2.5]{HH11})
the Alexander dual $J(T)$ defines a shellable simplicial complex.
\end{proposition}

\begin{proof}
Let  $r$ be any vertex in $T$ and orient all the arrows away from $r$. This gives
a partial order on the vertices of the tree. Take a linear extension of this
partial order and let $I$ be generated by the monomials in an initial segment
for this total order. Let $m_u$ be the subsequent monomial. We show that $I:m_u$ is 
generated by variables. So let $m \in I:m_u$. Then $m \cdot m_u$ is divisible by
some $m_v \in I$. Each of $m_u$ and $m_v$ have one variable for each edge in $T$. Only the
edges on the path between $u$ and $v$ give distinct variables in $m_u$ and $m_v$.
Starting from $u$ let $e$ be the first edge on the path to $v$. Let $w$ be the other end point of $e$.
Since $m_v$ comes before $m_u$ in the total order, we see that $m_w$ must also be in the initial segment
which generates $I$.
Of the variables in $m_u$ and $m_w$, only the $e$-variable is different. Thus
$I:m_u$ contains
the $e$-variable $x_{l(e),e_{to}(w)}$ occurring in $m_w$. But then this $e$-variable is also
in $m_v$ and so this variable divides $m$.
\end{proof}

\section{Alexander Duals} \label{sec:AD}

Recall that for a squarefree monomial ideal $I$ in a polynomial ring $S$, the 
{\it Alexander
dual} ideal $J$ is the monomial ideal in $S$ whose monomials are precisely those
that have nontrivial common divisor with every monomial in $I$, or equivalently, every generator of $I$.

We describe the Alexander dual ideal of any polarization of the power
$(x_1,\ldots, x_m)^n$. The description is a direct construction involving the
isotone maps $X_i, i = 1, \ldots, m$. The proof is reduced to study 
certain baby-versions $\chi_i$ of the $X_i$ which are isotone maps
$\chi_i : \Delta_m(n) \pil \{0 < 1\}$.

\subsection{Statement and examples}
Let $J
\subset k[\Xv_1,\ldots,\Xv_m]$ be a polarization of the ideal $(x_1,\ldots,x_m)^n$ in $k[x_1,\ldots,x_m]$.  For any $\mathbf{a}\in\Delta_m(n-1)$ we have the up-graph $U(\bfa)$,
with vertices $\bfa + e_j$ for $j = 1, \ldots, m$. At the vertex $\bfa + e_j$ we have the
$x_j$-type variables $X_j(\bfa + e_j)$. We take the product of all these variable sets:
\[ \MM(\mathbf{a})= \prod_{j = 1}^m X_j(\bfa + e_j). \]
It consists of monomials $x_{1i_1}x_{2i_2} \cdots x_{mi_m}$ where
$x_{ji_j}$ is in $X_j(\bfa + e_j)$.  
Let $I$ be the ideal generated by the monomials in the union of all the $\MM(\bfa)$ for
$\bfa \in \Delta_m(n-1)$.

\begin{theorem} \label{thm:ADIJ}
  The ideal $I$ is the Alexander dual of $J$. 
\end{theorem}

\begin{example}
Consider again the polarization $J$ from Example \ref{eks:LS-m3}. Its graph of linear syzygies is again given in Figure \ref{fig:AD-xyz}, this time labeling each of the up-triangles in the graph with
$1,\ldots, 6$.
\begin{figure}
\begin{tikzpicture}[]

\draw  (0, 4.33) -- (-2.5,0)--(2.5,0)--(0,4.33) ;

\draw[dashed]  (-1.68, 1.44)--(0,1.44);

\draw (0,1.44)--(1.68, 1.44) ;

\draw[dashed] (-0.85,2.88)--(0,1.44);

\draw (0,1.44)--(0.85,2.88)--(-0.85,2.88) ;

\draw (-1.68,1.44)--(-0.84,0)--(0,1.44);

\draw (0,1.44)--(0.84,0);

\draw[dashed] (0.84,0)--(1.68,1.44);

\draw (-1.5,2.93) node {$x_1x_2y_2$};
\draw (1.5, 2.93) node {$x_1x_2z_2$};
\draw (0.05,4.6) node   {$x_1x_2x_3$};
\draw (-3.3,0) node   {$y_1y_2y_3$};
\draw (3.3,0) node   {$z_1z_2z_3$};
\draw (2.4,1.5) node   {$x_1 z_1z_2$};
\draw (-2.4,1.5) node   {$x_2y_1y_2$};
\draw (-0.84,-0.3) node {$y_1y_2z_2$};
\draw (0.84, -0.3) node {$y_1z_2z_3$};
\draw (0.75, 1.25) node {$x_1y_1z_2$};

\draw (0,3.5) node {1};
\draw (-0.84, 2) node {2};
\draw (0.84, 2) node {3};
\draw (-1.68, 0.5) node {4};
\draw (0,0.5) node {5};
\draw (1.68, 0.5) node {6};

\end{tikzpicture}
\caption{}
\label{fig:AD-xyz}
\end{figure}
The up-triangle $i$ corresponds to an element $a_i \in \Delta_3(2)$, where $a_1 = (2,0,0), a_2 = (1,1,0), a_3 = (1,0,1), a_4 = (0,2,0), a_5 = (0,1,1)$, and $a_6 = (0,0,2)$.
The sets of monomials $\mathcal{M}(\mathbf a_i)$ are:

\begin{align*}
    \mathcal{M}(\mathbf a_1) &= \{\mathbf{x_1y_2z_2, x_2y_2z_2,x_3y_2z_2}\}\\
    \mathcal{M}(\mathbf a_2) &= \{\mathbf{x_1y_1z_2}, x_1y_2z_2,\mathbf{x_2y_1z_2}, x_2y_2z_2\}\\
    \mathcal{M}(\mathbf a_3) &= \{\mathbf{x_1y_1z_1}, x_1y_1z_2,\mathbf{x_2y_1z_1}, x_2y_1z_2\}\\
    \mathcal{M}(\mathbf a_4) &= \{x_2y_1z_2, x_2y_2z_2,\mathbf{x_2y_3z_2}\}\\
    \mathcal{M}(\mathbf a_5) &= \{x_1y_1z_2, \mathbf{x_1y_1z_3},x_1y_2z_2, \mathbf{x_1 y_2 z_3}\}\\
    \mathcal{M}(\mathbf a_6) &= \{x_1y_1z_1, x_1y_1z_2,x_1y_1z_3\}\\
\end{align*}

The boldface monomials are the ten distinct monomials we find from this process, which in fact generate the Alexander dual $I$ of $J$.

\end{example}

We shall go through several steps in proving the above theorem.
It turns out that we will be able to abstract the situation so our
arguments will only involve a collection of isotone maps
\begin{equation} \label{eq:AD-chi} \chi_i : \Delta_m(n) \pil \{ 0 < 1\},
\quad i = 1, \ldots, m
\end{equation}
where $\Delta_m(n)$ has the partial order $\geq_i$, and such that 
$\chi_i(\bfb) = 0$ {\it whenever} $b_i = 0$.

\begin{remark} When $m=3$, it is a curious fact that for each $i$ the number
of maps $\chi_i$ is the Catalan number $C_{n+1}$: Such maps are in one-one 
correspondence with "stacking of coins", \cite[Exercise 6.19 hhh]{St-EC}.
\end{remark}

First we establish some notation.
For a monomial $\mathbf m\in k[\Xv_1, \ldots, \Xv_m]$, define maps
\begin{align*}
\chi_{i,\mathbf{m}}& : \Delta_m(n)\rightarrow \{0 < 1\} \\ 
\mathbf{b} & \mapsto
\begin{cases}
0, & \text{no variable of } X_i(\mathbf b) \text{ is in } \bfm.\\
1, & \text{some variable of } X_i(\mathbf b) \text{ is in } \bfm.
\end{cases}
\end{align*}

Note i) this is an isotone map and ii) if $(\mathbf c; j, k)$ is a linear syzygy edge
for the isotone maps $\{X_i\}$,
then $\chi_{i,\bfm}(\mathbf c-e_j) = \chi_{i,\bfm}(\mathbf c-e_k)$ for every $i\neq j,k$.
Note furthermore:
\begin{itemize}
\item That a monomial $\mathbf m \in k[\Xv_1,\ldots,\Xv_m]$ is in the Alexander dual of $J$
means that it has a common variable with every $m(\bfb)$, or for every $\bfb \in \Delta_m(n)$,
a common variable with some $m_i(\bfb)$ for $i = 1, \ldots, m$.
This holds if and
only if for every such $\bfb$, there is some $i$  with $\chi_{i,\bfm}(\bfb) = 1$.
\item The  monomial $\bfm$ is in $I$ if and only if for some $\bfa \in \Delta_m(n-1)$,
it has a common variable with every  $X_j(\bfa + e_j)$ for $j = 1, \ldots, m$.
Thus for such an $\bfa$ we have $\chi_{j,\bfm}(\bfa + e_j) = 1$ for every $j$.
\end{itemize}

We now abstract the negations of the above.
Consider for $i = 1, \ldots, m$ isotone maps as in \eqref{eq:AD-chi}.

\begin{definition} A multidegree $\mathbf{b}\in\Delta_m(n)$ is a \textit{full zero point} for the collection $\{\chi_i \}$ if $\chi_{i}(\mathbf{b}) = 0$ for every 
    $i = 1, \ldots, m$. 
An up-simplex $U(\mathbf a)$ of $\Delta_m(n)$ {\it has a zero corner} if $\chi_{i}(\mathbf a + e_i) = 0$ for some $i$.
   
   An edge $(\bfc;i,j)$ of $\Delta_m(n)$ is a {\it linear syzygy edge} for the collection $\{ \chi_i \}$  
   if $\chi_p(\bfc -e_i) = \chi_p(\bfc - e_j)$
   for every $p \neq i,j$. 
   \end{definition}

We prove the following.

\begin{theorem} \label{thm:AD-chi} Given the collection of isotone maps $\{ \chi_i \}$ such that for every down-graph of $\Delta_m(n)$
the linear syzygy edges for $\{ \chi_i \}$ contains a spanning tree. Then $\{ \chi_i \}$ has a full zero-point in $\Delta_m(n)$ if and only if 
every up-graph of $\Delta_m(n)$ has a zero corner.
\end{theorem}

As a consequence we get Theorem \ref{thm:ADIJ}.

\begin{proof}[Proof of Theorem \ref{thm:ADIJ}]
That $\bfm \notin I$ means that every up-graph in $\Delta_m(n)$ has a zero corner for the $\chi_{i,\bfm}$'s.
That $\bfm$ is not in the Alexander dual of $J$ means that it has a full zero-point
for the $\chi_{i,\bfm}$'s. Hence by Theorem \ref{thm:AD-chi}
$I$ will be the Alexander dual of $J$.
\end{proof}

\subsection{Definitions and key lemma} \label{sec:ADDefKey}
In order to facilitate our arguments we need to have a more flexible framework
to work in.
For $\bfd \in \NN^m$ with $|\bfd| \leq n$ and $S \sus [m]$,
let $\Delta_S(n,\bfd)$ be the induced subgraph of $\Delta_m(n)$
whose vertices are the degrees
$\bb \geq \bfd$ such that
$\supp(\bb - \bfd) \sus S$.
This means:
\begin{equation*}
i)\,\, |\bfb| = n, \quad
ii) \,\, b_i = d_i \text{ for } j \in [m] \setminus S, \quad
iii) \,\, b_i \geq d_i \text{ for } i \in S.
\end{equation*}
(We omit $m$ in $\Delta_S(n,\bfd)$ since $m$ is fixed througout.) 
We normally write $\Delta_S(\bfd)$ for $\Delta_S(n,\bfd)$ but on a
few occasions we want $(n-1)$ or $(n+1)$ instead of $n$ as argument, and use the full notation.

The convex hull of $\Delta_S(\bfd)$ in $\RR^S$ is a simplex of dimension
$|S|-1$ if $|\bfd| < n$. The {\it size} of this simplex is $n-|\bfd|$.
For $\bfd = {\mathbf 0}$, the
zero degree, we have $\Delta_{[m]}({\mathbf 0}) = \Delta_m(n)$ as defined earlier.
Note that $\Delta_{S^\prime}(\bfd^\prime)$ is a non-empty subset of
$\Delta_S(\bfd)$ iff:
\begin{equation*}
i)\,\, S^\prime \sus S, \quad ii) \,\, \bfd^\prime \geq \bfd, \quad
iii) \,\, d^\prime_i = d_i  \text{ for } i \in [m] \setminus S.
\end{equation*}

\begin{example}\label{ex: subgraph} Let $S = \{2,3,4\}$ and $\bfd = (1,0,0,0)$. Then $\Delta_S(\bfd)$ is the induced subgraph of $\Delta_4(3)$ depicted in Figure \ref{fig:exSubgraph}; its convex hull is a simplex of dimension $2$, and its size is $2$. If $S^\prime = \{2,4\}$, then $\Delta_{S^\prime}(\bfd)$ is the subgraph of $\Delta_S(\bfd)$ depicted by the thick line in Figure \ref{fig:exSubgraph}. Its convex hull is a simplex of dimension $1$, and its size is also $2$.
\begin{figure}
\begin{tikzpicture}
\draw (-2,1.6)--(2,1.6);
\draw (-1,3.2)--(1,3.2);
\draw (-2,1.6)--(0,4.8);
\draw[very thick] (2,1.6)--(0,4.8) ;
\draw (0,1.6)--(1,3.2);
\draw (0,1.6)--(-1,3.2);
\filldraw[black] (0,4.8) circle (2pt)  node[anchor=south] at (0,4.9){(1,2,0,0)};
\filldraw[black] (-2,1.6) circle (2pt)  node[anchor=east] at (-2.1,1.6){(1,0,2,0)};
\filldraw[black] (2,1.6) circle (2pt)  node[anchor=west] at (2.1,1.6){(1,0,0,2)};
\end{tikzpicture}
\caption{}
\label{fig:exSubgraph}
\end{figure}
\end{example}

\begin{definition}
For $\aa \in \Delta_S(n-1,\bfd)$ we get an induced subgraph 
$U_S(\bfa;\bfd)$ of $\Delta_S(\bfd)$
with vertices $\{ \aa + e_i \, | \, i \in S \}$.
This is a complete graph on $|S|$ vertices whose convex hull is of dimension $|S|-1$.
The graph $U_S(\aa;\bfd)$ is an {\it up-graph}.
The $i \in S$ are the {\it corners} of the up-graph.

For $\cc \in \Delta_S(n+1,\bfd)$ we get an induced subgraph 
$D_S(\bfc;\bfd)$ of $\Delta_S(\bfd)$
with vertices $\{ \cc - e_i \, | \,  i \in \supp(\cc - \bfd) \}$.
This is a complete graph on $|\supp(\bfc-\bfd)|$
vertices whose convex hull is a simplex of dimension $|\supp(\cc - \bfd)|-1$. 
The graph $D_S(\bfc;\bfd)$ is a {\it down-graph}.
\end{definition}

\medskip

Suppose that for each $j \in S$ we have isotone maps, where we have
given $\Delta_S(\bfd + e_j)$ the $\geq_j$-ordering:
\[ \chi_j : \Delta_S(\bfd + e_j) \pil \{ 0 < 1\} .\]
Note that for $\bfb \in \Delta_S(\bfd)$, the $\chi_j(\bfb)$ are defined precisely 
for $j \in \supp(\bfb - \bfd)$.

\begin{definition}
A vertex $\bfb \in \Delta_S(\bfd)$ is a {\it full zero point} for 
$\{ \chi_i \}_{i \in S}$ if $\chi_i(\bfb) = 0$ for every $i \in \supp(\bfb - \bfd)$.
An up-graph $U_S(\bfa;\bfd)$ {\it has a zero corner} for $\{ \chi_i\}_{i \in S}$ if
$\chi_j(\bfa + e_j) = 0$ for some $j \in S$. 

An edge $(\cc;r,s)$ of $\Delta_S(\bfd)$ (note that then $\{r,s \} \sus S$)
is a {\it linear syzygy edge} for the $\{ \chi_j\}_{j \in S}$ if
\[ \chi_j(\cc - e_r) ) = \chi_j(\cc - e_s), \quad \text{for }
  j \in \supp(\cc - \bfd) \setminus \{r,s \}. \]
\end{definition}


\begin{lemma} \label{lem:ADrestrict} Suppose we have isotone maps
$\{\chi_j \}_{j \in S}$ for $\Delta_S(\bfd)$ such that the
linear syzygy edges in every down-graph of $\Delta_S(\bfd)$
contains a spanning tree.

Let $\Delta_{R}(\bfd^\prime)$ be a non-empty subgraph of $\Delta_S(\bfd)$.
For $j \in R$ let $\uchi_j$ be the restriction of $\chi_j$
to isotone maps associated to $\Delta_R(\bfd^\prime)$.
Then each down-graph of
$\Delta_R(\bfd^\prime)$ contains a spanning tree of linear syzygy edges for the
$\{\uchi_j\}_{j \in R}$. 
\end{lemma}

\begin{proof} (This proof is essentially the same as for Lemma \ref{lem:lin-R}.)
Choose a particular down-graph $D_R(\cc;\bfd^\prime)$ of $\Delta_R(\bfd^\prime)$.
Let $\cc - e_r$ and $\cc - e_s$ be vertices of $D_R(\cc;\bfd^\prime)$
(so $r,s \in \supp (\cc - \bfd^\prime)$).
We will show that there is a path from $\cc - e_r$ to $\cc - e_s$
in the graph $D_R(\cc;\bfd^\prime)$ consisting of linear syzygy edges
for the $\{\uchi_j\}_{j \in R}$.

Since $\cc \geq \bfd$, $|\cc| = n+1$ and $c_i = d_i$
for $i$ outside of $S$, we may consider the down graph
$D_S(\cc;\bfd)$ in $\Delta_S(\bfd)$.
We know there is a path in $D_S(\cc;\bfd)$ from $\cc - e_r$
to $\cc - e_s$
consisting of linear syzygy edges for the $\{\chi_j \}_{j \in S}$. 

It may be broken up
  into smaller paths: From $\bfc - e_r = \bfc- e_{r_0}$ to $\bfc-e_{r_1}$,
  from $\bfc - e_{r_1}$ to $\bfc - e_{r_2}$, ...,  from
  $\bfc- e_{r_{p-1}}$ to $\bfc - e_{r_p} = \bfc - e_s$
  where on the path from $\bfc - e_{r_{i-1}}$ to $\bfc - e_{r_i}$
  the only vertices  $\bfc -e_q$ with  $q \in R$ are the end vertices
  $q = r_{i-1}$ and $q = r_i$ while the 
  in between vertices all have $q \in Q = \supp(\cc - \bfd) \setminus R$.
  We claim that each edge from $\bfc - e_{r_{i-1}}$ to $\bfc - e_{r_i}$ is a linear syzygy
  edge for the $\{\uchi_j\}_{j \in R}$. This will prove the lemma.

  Let the path from $\bfc - e_{r_{i-1}}$ to $\bfc - e_{r_i}$ be
  \[ \bfc - e_{r_{i-1}} = \bfc - e_{q_0}, \bfc - e_{q_1}, \ldots, \bfc - e_{q_{t-1}}, \bfc - e_{q_t}  = \bfc - r_{i} \]
  where $q_1, \ldots, q_{t-1}$ are all in $Q$. We must show that
  \begin{equation} \label{eq:lin-Rlin}
    \uchi_p(\bfc - e_{r_{i-1}}) = \uchi_p(\bfc - e_{r_i}) \text{ for }
    p  \in  \supp(\cc - \bfd^\prime) \setminus \{ r_{i-1}, r_i \}.
    \end{equation}
    But since the edges on the path are linear syzygy edges for the
    $\{\chi_j\}_{j \in S}$
    we have
  \[ \chi_p(\bfc - e_{q_{j-1}}) = \chi_p(\bfc - e_{q_j}) \text{ for }
    p \in \supp(\bfc - \bfd) \setminus \{ q_{j-1}, q_j \}.\]
Hence    
  \[ \chi_p(\bfc - e_{r_{i-1}}) = \chi_p(\bfc - e_{r_i}) \text{ for }
    p \in \supp(\bfc - \bfd) \setminus \{ r_{i-1}, q_1, \ldots, q_{t-1},
  r_i\}.\]
Since $\supp(\cc- \bfd^\prime) \sus R$ and $\{q_1, \ldots, q_{t-1} \}
\sus S \setminus R$, these are disjoint sets. Also
$\supp (\bfc - \bfd^\prime) \sus \supp(\cc - \bfd)$. 
Hence
 \[ \chi_p(\bfc - e_{r_{i-1}}) = \chi_p(\bfc - e_{r_{i}}) \text{ for }
    p \in \supp(\bfc - \bfd^\prime) \setminus \{ r_{i-1}, r_i\}.\] 
  This shows \eqref{eq:lin-Rlin}.
  \end{proof}

\subsection{Full zero point implies a zero corner in every up-graph}
We prove this direction of Theorem \ref{thm:AD-chi}.

\begin{lemma}\label{lem:ADmVarZeroes} Let $S \sus [m]$ have cardinality $\geq 2$.
Suppose every down-graph of $\Delta_S(\bfd)$ contains a spanning tree of linear syzygies
for the isotone maps $\{\chi_j \}_{j \in S}$. Let $\bfb \in \Delta_S(\bfd)$ and $p \in S$.
Suppose $\chi_i(\bfb) = 0$ for every $i \in \supp(\bfb - \bfd) \setminus \{p \}$. Then every up-graph
$U_S(\bfa;\bfd)$ in $\Delta_S(\bfd)$ with $a_p \geq b_p$ 
has a zero corner for the $\{\chi_i\}_{i \in S}$.
\end{lemma}

\begin{corollary} \label{cor:ADPtoU}
  Suppose the maps $\{ \chi_i \}_{i \in S}$ have a full zero-point in $\Delta_S(\bfd)$. Then
  every up-graph in $\Delta_S(\bfd)$ has a zero corner.
\end{corollary} 

 \begin{proof}[Proof Lemma \ref{lem:ADmVarZeroes} and Corollary \ref{cor:ADPtoU}]
 We prove these in tandem by induction on the cardinality $|S|$. We prove Corollary
 \ref{cor:ADPtoU} by assuming Lemma \ref{lem:ADmVarZeroes}. Then we prove Lemma \ref{lem:ADmVarZeroes}
 by assuming Corollary \ref{cor:ADPtoU} has been proven for all $\Delta_{S^\prime}(\bfd)$ 
 when the cardinality $|S^\prime| < |S|$.
 
 \medskip
 Assume we have shown Lemma \ref{lem:ADmVarZeroes}. 
 Let $\bfb$ be a full zero-point for $\Delta_S(\bfd)$ and 
 consider an up-graph $U_S(\bfa;\bfd)$ in $\Delta_S(\bfd)$. Since $|\bfb| = n$ and $|\bfa| = n-1$, 
 then at least for one $p$ we have
   $a_p \geq b_p$. Then Lemma \ref{lem:ADmVarZeroes} implies
  that $U_S(\bfa;\bfd)$ is an up-graph with a zero corner, proving Corollary \ref{cor:ADPtoU}.

\medskip We now show Lemma \ref{lem:ADmVarZeroes}. For simplicity we assume $p = 1$.
First we do the case $|S| = 2$, 
say $S = \{1,2 \}$. Then $\chi_2(\bfb) = 0$. By isotonicity of
$\chi_2$ we have $\chi_2(\bfb + \lambda e_1 - \lambda e_2) = 0$ for $\lambda \geq 0$.
But letting $\lambda = a_1 - b_1$, the point  $\bfb + \lambda e_1 - \lambda e_2 = \bfa + e_2$ is a zero corner
for $U_S(\bfa;\bfd)$.

\medskip Assume now $|S| \geq 3$ and Corollary \ref{cor:ADPtoU} holds for $S^\prime$ with
$2 \leq |S^\prime| < |S|$. We argue by induction on the difference $a_1 - b_1$.

\medskip 
\noindent {\bf Case $a_1 = b_1$.} Let $S^\prime = S\setminus \{1 \}$ and $\bfd^\prime = 
\bfd + (b_1 - d_1)e_1$. Then 
$\bfb \in \Delta_{S^\prime}(\bfd^\prime)$ and $\bfb$ is a full zero point for the $\{ \uchi_i\}_{i \in S^\prime}$. 
By Corollary \ref{cor:ADPtoU} every up-graph $U_{S^\prime}(\bfa;\bfd^\prime)$ in $\Delta_{S^\prime}(\bfd^\prime)$ 
has a zero corner for $\{ \uchi_i\}_{i \in S^\prime}$. Since $\bfa \in \Delta_{S^\prime}(n-1,\bfd^\prime)$
iff $\bfa \in \Delta_S(n-1,\bfd)$ with $a_1 = d^\prime_1 = b_1$, up-graphs $U_{S^\prime}(\bfa;\bfd^\prime)$
in $\Delta_{S^\prime}(\bfd^\prime)$ correspond to up-graphs $U_S(\bfa;\bfd)$ with $a_1 = b_1$.
Then every up-graph $U_S(\bfa;\bfd)$ with $a_1 = b_1$ has a zero corner
for $\{ \chi_i \}_{i \in S}$.

\medskip
\noindent {\bf Case $a_1 > b_1$.}
   Consider the down-graph $D_S(\bfb + e_1)$. By assumption on
   the maps $\chi_i$, at least one edge $(\mathbf b+e_1; 1, r)$ is a linear
   syzygy edge, so $\chi_j(\mathbf b + e_1-e_r) = \chi_j(\mathbf b) = 0$ for every
   $j \in \supp(\bfb - \bfd) \setminus \{1, r\}$. But $\mathbf b + e_1-e_r\leq_r \mathbf b$,
   so $\chi_r(\mathbf b+e_1-e_r) \leq \chi_p(\mathbf b) = 0$.
   Therefore $\chi_j(\mathbf \bfb + e_1-e_r) = 0$
   for every $j \in \supp(\bfb - \bfd) \setminus \{ 1 \}$. 
   Since the difference in first coordinates of $\bfa$ and $\bfb + e_1 - e_r$ 
   is one less than $a_1 - b_1$, by induction the up-graph $U_S(\bfa;\bfd)$ has a zero corner
   for $\{ \chi_j \}_{j \in S}$.
 \end{proof}

\subsection{Every up-graph has a zero corner implies a full zero point}
We now prove the other direction of Theorem \ref{thm:AD-chi}.
We need the following specific lemma. It says that you can ``pull
a point'' with specific properties in a given direction and into a simplex
of smaller size.

\begin{lemma} \label{lem:ADpull} Let $\bfb \in \Delta_S(\bfd)$ with 
$|\bfd| \leq n-1$, $1 \in S$, and $b_1 \leq d_1 + 1$. Suppose 
\[ \chi_i(\bfb) = 0 \text{ for } i \in \supp(\bfb- \bfd) \setminus \{1\},
\quad \chi_1(\bfb) = 1 \text{ if } 1 \in \supp(\bfb - \bfd). \]
Let $p \in S \setminus \{1\}$ and $\bfd^\prime = \bfd + e_p$. 
Then there is some $\bfb^\prime \in \Delta_S(\bfd^\prime)$ with 
$b^\prime_1 \leq d^\prime_1 + 1 (= d_1 + 1)$ such that 
\begin{equation} \label{eq:ADpull}  \chi_i(\bfb^\prime) = 0 \text{ for } i \in \supp(\bfb^\prime- \bfd^\prime) \setminus \{1\},
\quad \chi_1(\bfb^\prime) = 1 \text{ if } 1 \in \supp(\bfb^\prime - \bfd^\prime). 
\end{equation}
\end{lemma} 
 
 \begin{proof}
If $b_p \geq  d_p +1$ we simply let $\bfb^\prime = \bfb$.
So suppose $b_p = d_p$. Looking at the down-graph $D_S(\bfb + e_p;\bfd)$
there is a linear syzygy edge $(\bfb + e_p;p,q)$ for some $q \in S \setminus \{p\}$.
It goes from $\bfb$ to $\bfb^\prime = \bfb + e_p - e_q$. Let us show that this
$\bfb^\prime$ has the desired properties.

We have $\bfb^\prime - \bfd^\prime = \bfb - \bfd - e_q$ and so
$\supp(\bfb^\prime - \bfd^\prime)$ is $\supp(\bfb - \bfd)$ with $q$ possibly 
removed.
We see that $p \not \in \supp(\bfb^\prime -\bfd^\prime)$,
and so $\chi_i(\bfb^\prime) = \chi_i(\bfb)$ for $i \in \supp(\bfb^\prime - \bfd^\prime) 
\setminus \{q\}.$

\medskip
\noindent {\bf Case $q = 1$.} In this case $b^\prime_1 = d_1$ so $1 \not \in \supp(\bfb^\prime - \bfd^\prime)$
and \eqref{eq:ADpull} holds.

\medskip 
\noindent {\bf Case $ q \neq 1$.} Note $\bfb^\prime \leq_q \bfb$. So if $q$ is contained in $\supp(\bfb^\prime - \bfd^\prime)$ then 
$\chi_q(\bfb^\prime) \leq \chi_q(\bfb) = 0$. Furthermore $\chi_1(\bfb^\prime) 
= \chi_1(\bfb) = 1$ if $1 \in \supp(\bfb^\prime - \bfd^\prime)$.
So again \eqref{eq:ADpull} holds.
\end{proof}

\begin{proposition} Let $|\bfd| \leq n-1$ and $|S| \geq 2$.
Suppose every up-graph in $\Delta_S(\bfd)$ has a zero corner for the $\{\chi_i\}_{i \in S}$. Then
there is an element of $\Delta_S(\bfd)$ which is a full zero-point. 
\end{proposition}

\begin{proof}
We prove by induction on the size $(n-|\bfd|)$ and cardinality $|S|$ that the above holds.
If $|\bfd| = n-1$, then $\Delta_S(\bfd)$ equals the up-graph $U_S(\bfd;\bfd)$. 
If we have a zero at corner $p \in S$, so $\chi_p(\bfd + e_p) = 0$, then $\bfb = \bfd + e_p$
is a full zero point of $\Delta_S(\bfd)$ since $\supp(\bfb - \bfd) = \{p \}$.

Now pick an element of $S$, say $1 \in S$. By induction on size, $\Delta_S(\bfd + e_1)$ has
a full zero point $\bfb$. If $b_1 > d_1 + 1$ then
$\supp(\bfb - \bfd) = \supp(\bfb -(\bfd + e_1))$ and $\bfb$ is a full zero
point for $\Delta_S(\bfd)$. So suppose $b_1 = d_1 + 1$. If $\chi_1(\bfb) = 0$,
it is a full zero point in $\Delta_S(\bfd)$. Otherwise we have:
\[ \chi_i(\bfb) = 0 \text{ for } i \in \supp(\bfb - \bfd) \setminus \{1\}, \quad
\chi_1(\bfb) = 1.\]

First consider when $S$ has cardinality two, say $S = \{1,2\}$, let $\bfa = \bfb - e_1$ and consider the up-graph $U_S(\bfa;\bfd)$.
Since $\chi_1(\bfa + e_1) = 1$, we must in the other corner of this up-graph have $\chi_2(\bfa + e_2) = 0$.
Since $1$ is not in the support of $(\bfa + e_2) - \bfd$, then $\bfb = \bfa + e_2$ is a full zero point.

\medskip
So let $S$ have cardinality $\geq 3$ and put
$S^\prime = S \setminus \{ 1\}$.
We show now that each up-graph in $\Delta_{S^\prime}(\bfd)$ has a zero corner. By induction 
we then have a full zero $\bfb^\fus$ in $\Delta_{S^\prime}(\bfd)$. Since $1$ is 
not in the support of $\bfb^\fus - \bfd$ this $\bfb^\fus$ is also a full zero point in $\Delta_S(\bfd)$ and we are done.

\medskip 
So let $U_{S^\prime}(\bfa;\bfd)$ be an up-graph in $\Delta_{S^\prime}(\bfd)$. We also have
the up-graph $U_S(\bfa;\bfd)$ in $\Delta_S(\bfd)$. If there is some
$p \geq 2$ such that $a_p > b_p$ we apply Lemma \ref{lem:ADpull} and "pull" the point $\bfb$  to a 
point $\bfb^\prime$ in a smaller sized $\Delta_S(\bfd + e_p)$, which still contains $U_S(\bfa;\bfd)$.
In this way we continue until we have a $\bfb$ with either i) $a_p \leq b_p$ for every $p \geq 2$,  or ii) the size of $\Delta_S(\bfd)$ has become
$1$. But with this size we have $\bfa = \bfd$ and so $a_p \leq b_p$ for every $p \geq 2$ in any case. 
We also have $a_1 = d_1 \leq b_1 \leq d_1 + 1$ and recall that $|\bfb| = n$ and $|\bfa| = n-1$. 

\medskip 
\noindent {{\bf Case $b_1 = a_1$.}} Then $\bfb = \bfa + e_p$ for some $p \geq 2$ and $p \in \supp(\bfb - \bfd)$.
Then  
\[ \chi_p(\bfa + e_p) = \chi_p(\bfb) = 0 \]
and so $U_{S^\prime}(\bfa;\bfd)$ has a zero corner.

\medskip
\noindent {{\bf Case $b_1 = a_1 +1$.}} Then $\bfb = \bfa + e_1$ and $1 \in \supp(\bfb - \bfd)$
so 
\[ \chi_1(\bfa +e_1) = \chi_1(\bfb) = 1. \]
Since $U_S(\bfa;\bfd)$ has a zero corner we must have $\chi_p(\bfa + e_p) = 0$ for some $p \in S^\prime$
and so $U_{S^\prime}(\bfa;\bfd)$ has a zero corner.
\end{proof}

\section{Polarizations define shellable simplicial complexes}
\label{sec:LQ}

In this section we show that the Alexander dual of any polarization of the power $(x,y,z)^m$ is a monomial ideal with linear quotients.
This is equivalent to the polarization defining a shellable
simplicial complex via the Stanley-Reisner correspondence, \cite[Prop.8.2.5]{HH11}. By a result of Bj\"orner \cite[Thm.11.4]{Bj95} this immediately
implies that these polarizations define simplicial balls, see 
Lemma \ref{lem:conj-codimone} and Subsection \ref{subsec:conj-ball}.

\medskip
An element $(a,b,c)$ in $\Del_3(n-1)$ corresponds to an up-triangle
in $\Del_3(n)$. If $x_\alpha \in X(a+1,b,c)$ we say that $x_\alpha$
(or just $\alpha)$ is an $x$-variable belonging to the up-triangle
$U(a,b,c)$. Similarly if $y_\beta \in Y(a,b+1,c)$ and $z_\gamma \in
Z(a,b,c+1)$. We also say the monomial $x_\alpha y_\beta z_\gamma$ (or
just $\alpha \beta \gamma$) belongs
to $(a,b,c)$.

\begin{lemma} \label{lem:lq-push} Suppose $\alpha \beta \gamma$ belongs to the
up-triangle $U(a+1,b,c)$ in $\Del_3(n-1)$, see Figure \ref{fig:LQ-U2}. 

a. Then either the up-triangle $U(a,b+1,c)$ or the up-triangle $U(a,b,c+1)$
has a monomial
$\alpha^\prime \beta \gamma$ belonging to them.

b. If $\alpha$ either belongs to the up-triangle $U(a,b+1,c)$ or to
$U(a,b,c+1)$, then $\alpha \beta \gamma$ will belong to one of the up-triangles
$U(a,b,c+1)$ or in $U(a,b+1,c)$.

\end{lemma}

\begin{proof}
Consider the up-triangles in Figure \ref{fig:LQ-U2}.
\begin{figure}
\begin{tikzpicture}
\draw[->](-1.2,2.4)--(0,2.4) node[anchor=east] at (-1.2,2.4){$U(a+1,b,c)$};
\draw[->](-2.2,0.8)--(0,0.8) node[anchor=east] at (-2.2,0.8){$D(a+1,b+1,c+1)$};
\draw (-2,0)--(2,0);
\draw (-1,1.6)--(1,1.6);
\draw (-2,0)--(0,3.2) node[anchor=north] at (0,3.1) {$\alpha$};
\draw (2,0)--(0,3.2) ;
\draw (0,0)--(-1,1.6) node[anchor=south west] at (-.9,1.5) {$\beta$};
\draw (0,0)--(1,1.6) node[anchor=south east] at (.9,1.5) {$\gamma$};
\draw (0,0)--(-1,1.6) node[anchor=south west] at (-1.9,0) {$\beta$};
\draw (0,0)--(1,1.6) node[anchor=south east] at (1.9,0) {$\gamma$};
\end{tikzpicture}
\caption{}
\label{fig:LQ-U2}
\end{figure}
In the middle we have a down-triangle $D(a+1,b+1,c+1) \in \Del_3(n)$.
Note that since $Y$ is isotone, $\beta$ will be in both the up-triangles $U(a+1,b,c)$ and $U(a,b+1,c)$
and since $Z$ isotone $\gamma$ in both $U(a+1,b,c)$ and $U(a,b,c+1)$.

a. If the edge $((a+1,b+1,c+1);1,2)$ is a linear syzygy edge, then
also $\gamma$ belongs to $U(a,b+1,c)$, and if 
$((a+1,b+1,c+1);1,3)$ is
a linear syzygy edge then $\beta$ belongs to $U(a,b,c+1)$. Since at least one of them is a linear syzygy edge we are done.

b. If $((a,b+1,c+1);2,3)$ is a linear syzygy edge, $\alpha$ is either in none
or in both the two lower up-triangles.
It then follows by part a that $\alpha \beta \gamma$ belongs
to one of these up-triangles.

If $((a,b+1,c+1);2,3)$ is not a linear syzygy edge, the two other
edges are linear syzygy edges. By the argument in
part a, both the lower up-triangles contains $\beta$ and $\gamma$ and
so at least on of them contains $\alpha \beta \gamma$.
\end{proof}

Let $\Xv$ be a set of $x$-variables (with various indices) and $\Yv$
and $\Zv$ be sets of $y$- and $z$-variables. 

\begin{lemma} \label{lem:lq-var}
 Let $I$ be an ideal
generated by a subset of monomials in the product set $\Xv\cdot \Yv \cdot 
\Zv$. Let $x_\alpha y_\beta z_\gamma$
be in $\Xv\cdot \Yv \cdot \Zv$ but not in $I$. Then $I: x_\alpha y_\beta z_\gamma$ is
generated by variables iff for every $x_{\alpha^\prime} y_{\beta^\prime}
z_{\gamma^\prime} \in I$ one of the variables $x_{\alpha^\prime}, y_{\beta^\prime}$
or $z_{\gamma^\prime}$ is in the colon ideal.
\end{lemma}

\begin{proof} Note that by the construction of $I$ and definition of
$x_\alpha y_\beta z_\gamma$, none of the
variables $x_\alpha, y_\beta$, or $z_\gamma$ can be in $I:x_\alpha y_\beta
z_\gamma$.

   That the first assertions implies the second is easy. Assume the
 second assertion holds. Then, if say $y_{\beta^\prime} z_{\gamma^\prime}$ is
 in the colon ideal, then $x_\alpha y_\beta y_{\beta^\prime} z_\gamma z_{\gamma^\prime}$ is
 in $I$. So at least some $x_\alpha y_{\tilde{\beta}} z_{\tilde{\gamma}}$ 
 is in $I$, where $\tilde{\beta} = \beta^\prime$ or $\tilde{\gamma} = \gamma^\prime$. But by assumption then either $y_{\beta^\prime}$ or
 $z_{\gamma^\prime}$ is in the colon ideal. This implies the colon ideal
 is generated by variables.
\end{proof}

We now consider the monomials $x_\alpha y_\beta z_\gamma$ belonging
to the up-triangles $U(a,b,c) \in \Del_3(n-1)$ and shall provide a total
order on these monomials.
First consider the partial order on {\it triples} where $(a,b,c) \geq (a^\prime,
b^\prime, c^\prime)$ if $a \geq a^\prime$ and take any linear extension
on this to get a total order $\succeq$ on triples.

Now for each up-triangle $U(a,b,c)$ we shall make a total order on
the (degree $3$) {\it monomials} belonging to it. 
For each $X(a+1,b,c)$ choose any total order of the $x$-variables.
To order the variables in $Y(a,b+1,c)$ we have an ascending chain
\begin{equation} \label{eq:lq-Y} 
Y(a,1,n-a-1) \sus Y(a,2,n-a-2) \sus \cdots \sus Y(a,n-a,0).
\end{equation}
We order the variables such that each new variable popping up in the
chain is less than the foregoing variables.
Similarly for the variables in $Z(a,b,c+1)$ we have a chain
\[  Z(a,n-a-1,1) \sus Z(a,n-a-2,2) \sus \cdots \sus Z(a,0,n-a), \]
and we order the variables such that each new variables popping up in the
chain is less than the foregoing variables.
The monomials belonging to $U(a,b,c)$ correspond to 
\[ X(a+1,b,c) \times Y(a,b+1,c) \times Z(a,b,c+1). \]
We get the partial product order on this 
and take a linear extension of this partial order.

We now order the monomials associated to the up-triangles in $\Del_3(n)$
as follows. If $\alpha^\prime \beta^\prime \gamma^\prime$ occurs
first in $(a^\prime, b^\prime, c^\prime)$ and $\alpha \beta \gamma$ occurs first
in $(a,b,c)$, then
\begin{equation} \label{eq:ShellOrder}
  \alpha^\prime \beta^\prime \gamma^\prime > \alpha \beta \gamma
\end{equation}
if $(a^\prime, b^\prime, c^\prime) \succ (a,b,c)$, or if
$(a^\prime, b^\prime, c^\prime) = (a,b,c)$ and the order of 
\eqref{eq:ShellOrder} is given
by the order on the monomials belonging to the up-triangle $U(a,b,c)$.

\begin{proposition}
The ideal generated by all the variables belonging to the up-triangles
of $\Del_3(n)$, has linear quotients given by the ordering of the monomials
above.
\end{proposition}

\begin{proof}
Let $\alpha \beta \gamma$ occur for the first time in the up-triangle
$(u,b,c)$ and let $I$ be the ideal generated by all the larger monomials.
We shall show that $I:x_\alpha y_\beta z_\gamma$ is generated by variables and
use Lemma \ref{lem:lq-var}.

\medskip
1. Let $\alpha^\prime \beta^\prime \gamma^\prime$ be in $(u,b^\prime,c^\prime)$
where $(u,b^\prime, c^\prime) \succ (u,b,c)$. Suppose $c^\prime \geq c$, see Figure \ref{fig:LQ-Uh}.
and so $\gamma$ belongs to $(u,b^\prime, c^\prime)$, since the map
$Z : \Delta_3(n) \pil B(n)$ is isotone.

\begin{figure}
\begin{tikzpicture}
\draw (-1,0)--(1,0);
\draw (-1,0)--(0,1.6);
\draw (1,0)--(0,1.6);
\draw[loosely dotted] (1.4,0.6)--(2.6,0.6);
\draw (3,0)--(5,0);
\draw (3,0)--(4,1.6);
\draw (5,0)--(4,1.6);
\draw node[anchor=north] at (0,1.5) {$\alpha$};
\draw node[anchor=south west] at (-.9,-0.1) {$\beta$};
\draw node[anchor=south east] at (.9,-0.1) {$\gamma$};
\draw node[anchor=north] at (4,1.5) {$\alpha^\prime$};
\draw node[anchor=south west] at (3.1,-0.1) {$\beta^\prime$};
\draw node[anchor=south east] at (4.8,-0.1) {$\gamma^\prime$};
\draw node[anchor=north] at (0,-0.1) {$(u,b,c)$};
\draw node[anchor=north] at (4,-0.1) {$(u,b^\prime,c^\prime)$};
\end{tikzpicture}
\caption{}
\label{fig:LQ-Uh}
\end{figure}

a. If $\beta \geq \beta^\prime$ then $\beta$ will be in
$U(u,b^{\prime \prime}, c^{\prime \prime})$ where $b^{\prime \prime}
\leq  b^\prime$ so $c^{\prime \prime} \geq c^\prime$.
Then $\beta$ will also be in $U(u,b^\prime, c^\prime)$ and
since $\gamma$ is in $U(u,b^\prime, c^\prime)$ we will have
$\alpha^\prime \beta \gamma$ belonging to $U(u,b^\prime, c^\prime)$.
If $\alpha = \alpha^\prime$ then $\alpha\beta\gamma$ would occur in
$U(u,b^\prime, c^\prime)$ contradicting that $\alpha\beta\gamma$ first
occurs in $U(u,b,c)$.
So $\alpha \neq \alpha^\prime$ and this gives $x_{\alpha^\prime} $ in the colon ideal.

b. Assume now that $\beta < \beta^\prime$. Note that since $b^\prime \leq b$
we have $\beta^\prime$ belonging to $U(u,b,c)$. 
Then $\alpha \beta^\prime
\gamma $ is already in $I$ by the ordering on the monomials belonging
to $U(u,b,c)$, and hence $\beta^\prime $ is in the colon ideal.

\medskip
2. A symmetric argument works when $\alpha^\prime \beta^\prime
\gamma^\prime $ is in $U(u,b^\prime, c^\prime)$ and $b^\prime \geq b$. 

\medskip
3. Assume now that $\alpha^\prime \beta^\prime \gamma^\prime$ belongs to
the up-triangle $U(u+1,b^\prime, c^\prime)$ where the sum of these
coordinates is $n-1$. Either $b^\prime \geq b$ or $c^\prime \geq c$.
Suppose the latter, see Figure \ref{fig:LQ-Une}.
Then $\beta^\prime$ belongs to the up-triangle $U(u,b,c)$ due to $Y$ being
isotone.

\begin{figure}
\begin{tikzpicture}
\draw (-1,0)--(1,0);
\draw (-1,0)--(0,1.6);
\draw (1,0)--(0,1.6);
\draw[loosely dotted] (1.4,1.6)--(2.6,1.6);
\draw (3,1.6)--(5,1.6);
\draw (3,1.6)--(4,3.2);
\draw (5,1.6)--(4,3.2);
\draw node[anchor=north] at (0,1.5) {$\alpha$};
\draw node[anchor=south west] at (-.9,-0.1) {$\beta$};
\draw node[anchor=south east] at (.9,-0.1) {$\gamma$};
\draw node[anchor=north] at (4,3.1) {$\alpha^\prime$};
\draw node[anchor=south west] at (3.1,1.5) {$\beta^\prime$};
\draw node[anchor=south east] at (4.8,1.5) {$\gamma^\prime$};
\draw node[anchor=north] at (0,-0.1) {$(u,b,c)$};
\draw node[anchor=north] at (4,1.5) {$(u+1,b^\prime,c^\prime)$};
\end{tikzpicture}
\caption{}
\label{fig:LQ-Une}
\end{figure}

a. If $\beta^\prime > \beta$ in the order given by \eqref{eq:lq-Y}, then 
$\alpha {\beta^\prime}\gamma >
\alpha \beta \gamma$ and so the former belongs to $I$ and 
$\beta^\prime$ is in the colon ideal $I : x_\alpha y_\beta z_\gamma$.

b. If $\beta^\prime =  \beta$ note that $y_\beta = y_{\beta^\prime}$ is in $Y(u,b-1,c+1)$
since $b > b^\prime$. By Lemma \ref{lem:lq-push} (applied in the $y$-direction, not $x$-direction) either $\alpha \beta \gamma$ is in
$U(u+1,b,c)$ or in $U(u,b-1,c+1)$. The latter must be the case since 
$\alpha \beta \gamma$ first occurs in $U(u,b,c)$.


c. If $\beta^\prime < \beta$, then since $\beta^\prime$ belongs to 
$U(u,b,c)$, $\beta$ must belong to $U(u,b-1,c+1)$ and by Lemma \ref{lem:lq-push}
$\alpha \beta \gamma$ will belong to $U(u,b-1,c+1)$.

In case b. and c. we may continue like this and push $\alpha \beta \gamma$
stepwise to the right, until we get to $(u,b-r,c+r)$ where 
 $c+r = c^\prime +1$, and
$(u,b-r,c+r) = (u,b^\prime, c^\prime + 1)$, so
$\alpha \beta \gamma$ is in both $U(u,b^\prime +1, c^\prime)$ and $U(u,b^\prime,
c^\prime + 1)$. Note that by $X$ being isotone $\alpha$ belongs to $(u+1,b^\prime,c^\prime)$.
We show that one of $\alpha^\prime, \beta^\prime$ or $\gamma^\prime$
is in the colon ideal.

\begin{itemize}
    \item[i.] If $\beta = \beta^\prime$ and $\gamma = \gamma^\prime$
    then $\alpha^\prime \beta \gamma$ belongs to $U(u+1,b^\prime, c^\prime)$.
    Since $\alpha \beta \gamma$ occurs first in $U(u,b,c)$, we cannot have 
    $\alpha = \alpha^\prime$ and so $x_{\alpha^\prime}$ is in the colon ideal.
    \item[ii.] If $\beta \neq \beta^\prime$ and $\gamma = \gamma^\prime$
    then $\alpha \beta^\prime \gamma$ belongs to $U(u+1,b^\prime, c^\prime)$
    and so $y_{\beta^\prime}$ is in the colon ideal.
    \item[iii.] If $\beta = \beta^\prime$ and $\gamma \neq \gamma^\prime$
    then $\alpha \beta \gamma^\prime$ belongs to $U(u+1,b^\prime,c^\prime)$
    and so $z_{\gamma^\prime}$ is in the colon ideal.
    \item[iv.] Suppose that $\beta \neq \beta^\prime$ and $\gamma \neq  
    \gamma^\prime$. If the edge $((u+1,b^\prime + 1, c^\prime +1);1,2)$ is a linear syzygy edge then $\gamma$ is in $Z(u+1,b^\prime, c^\prime + 1)$ and so $\alpha \beta^\prime \gamma$ belongs to $U(u+1,b^\prime, c^\prime)$.
If $((u+1,b^\prime + 1, c^\prime + 1);1,3)$ is a linear syzygy edge then
$\beta \in Y(u+1,b^\prime +1, c^\prime)$ and so $\alpha \beta \gamma^\prime$
is belongs to $U(u+1,b^\prime,c^\prime)$.
\end{itemize}

\medskip
4. Suppose then that $\alpha^\prime \beta^\prime \gamma^\prime$ is in
$U(u+r,b^\prime, c^\prime)$ where $r \geq 2$.

a. If $b^\prime \leq b$ and $c^\prime \leq c$ (then at least one inequality
is strict) then
$\alpha \beta^\prime \gamma^\prime$ is in $U(u+r,b^\prime, c^\prime)$
since the map $X$ is isotone,
and $\alpha$ also belongs to either $U(u+r-1,b^\prime +1, c^\prime)$ or
$U(u+r-1,b^\prime, c^\prime + 1)$. Hence  by Lemma \ref{lem:lq-push} $\alpha \beta^\prime
\gamma^\prime$ is in one of these up-triangles.
We may continue until either $u+r-1 = u+1$, treated in Case 3., or
until $b^\prime > b$ or $c^\prime > c$. Assume $c^\prime > c$.

b. We then assume $\alpha^\prime \beta^\prime \gamma^\prime$ is in
$(u+r,b^\prime, c^\prime)$ where $r \geq 2$ and $c^\prime > c$.
Note that by $Y$ being isotone and $b^\prime < b$, $\beta^\prime$ will belong to $U(u,b,c)$ and to $U(u,b-1,c+1)$. 

\begin{itemize}
\item[b1.] If $\beta^\prime > \beta$, then $\alpha \beta^\prime \gamma >
\alpha \beta \gamma$ and so $y_{\beta^\prime}$ is in the colon ideal.

\item[b2.] If $\beta = \beta^\prime$ then $\beta$ belongs to 
$U(u+r,b^\prime, c^\prime)$. By $Y$ being isotone, $\beta$ belongs to
$U(u,b-1,c+1)$.

\item[b3.] If $\beta^\prime < \beta$, then $\beta$ is in the up-triangle
$U(u,b-1,c+1)$. In both cases b2 and b3, by Lemma \ref{lem:lq-push} $\alpha \beta \gamma$ is either in up-triangle $U(u+1,b-1,c)$, not possible, or in $U(u,b-1,c+1)$.
\end{itemize}

In this way we may continue going rightwards until we get to
$(u,b-t,c+t)$ with $c+t =  c^\prime$. Then $(u,b^\prime + r, c^\prime)$
contains $\alpha \beta \gamma$ and so $\alpha $ is in $(u+r,b^\prime, c^\prime)$
and $(u+r-1,b^\prime +1, c^\prime)$. Then $\alpha \beta^\prime \gamma^\prime$
is in $(u+r,b^\prime, c^\prime)$ and since $\alpha \beta \gamma$ occurs 
first in $U(a,b,c)$ this is not equal to $\alpha \beta^\prime \gamma^\prime$. By Lemma \ref{lem:lq-push} we may 
push it down to level
$u+r-1$. In this way we can continue until we get $\alpha \beta^\prime
\gamma^\prime$ on level $u+1$ which is treated in Case 3.
\end{proof}

\input hoved.bbl
\end{document}

%% file: hoved.bbl
\providecommand{\bysame}{\leavevmode\hbox to3em{\hrulefill}\thinspace}
\providecommand{\MR}{\relax\ifhmode\unskip\space\fi MR }
\providecommand{\MRhref}[2]{%
  \href{http://www.ams.org/mathscinet-getitem?mr=#1}{#2}
}
\providecommand{\href}[2]{#2}